\pdfoutput=1
\RequirePackage{ifpdf}
\ifpdf 
\documentclass[pdftex]{sigma}
\else
\documentclass{sigma}
\fi

\numberwithin{equation}{section}

\newtheorem{Theorem}{Theorem}[section]
\newtheorem{Corollary}[Theorem]{Corollary}
\newtheorem{Lemma}[Theorem]{Lemma}
\newtheorem{Proposition}[Theorem]{Proposition}
 { \theoremstyle{definition}
\newtheorem{Definition}[Theorem]{Definition}
\newtheorem{Example}[Theorem]{Example}
\newtheorem{Remark}[Theorem]{Remark} }

\begin{document}
\allowdisplaybreaks

\newcommand{\arXivNumber}{1903.06770}

\renewcommand{\thefootnote}{}

\renewcommand{\PaperNumber}{086}

\FirstPageHeading

\ShortArticleName{The Ramificant Determinant}

\ArticleName{The Ramificant Determinant\footnote{This paper is a~contribution to the Special Issue on Algebraic Methods in Dynamical Systems. The full collection is available at \href{https://www.emis.de/journals/SIGMA/AMDS2018.html}{https://www.emis.de/journals/SIGMA/AMDS2018.html}}}

\Author{Kingshook BISWAS~$^\dag$ and Ricardo P\'EREZ-MARCO~$^\ddag$}

\AuthorNameForHeading{K.~Biswas and R.~P\'erez-Marco}

\Address{$^\dag$~Indian Statistical Institute, Kolkata, India}
\EmailD{\href{mailto:kingshook@isical.ac.in}{kingshook@isical.ac.in}}

\Address{$^\ddag$~CNRS, IMJ-PRG, University Paris 7, Paris, France}
\EmailD{\href{mailto:ricardo.perez-marco@imj-prg.fr}{ricardo.perez-marco@imj-prg.fr}}

\ArticleDates{Received March 13, 2019, in final form October 31, 2019; Published online November 05, 2019}

\Abstract{We give an introduction to the transalgebraic theory of simply connected log-Riemann surfaces with a finite number of infinite ramification points (transalgebraic curves of genus~0). We define the base vector space of transcendental functions and establish by elementary methods some transcendental properties. We introduce the Ramificant determinant constructed with transcendental periods and we give a closed-form formula that gives the main applications to transalgebraic curves. We prove an Abel-like theorem and a Torelli-like theorem. Transposing to the transalgebraic curve the base vector space of transcendental functions, they generate the structural ring from which the points of the transalgebraic curve can be recovered algebraically, including infinite ramification points.}

\Keywords{transalgebraic theory; Ramificant determinant; log-Riemann surface; Dedekind--Weber theory; ramified covering; exponential period; Liouville theorem}

\Classification{30F99; 30D99}

\renewcommand{\thefootnote}{\arabic{footnote}}
\setcounter{footnote}{0}

\vspace{-2mm}

\section{Introduction}

The authors defined the notion of log-Riemann (and tube-log, see \cite{bipmarxiv2}) surfaces in the seminal manuscript \cite{bipmarxiv} (see also \cite{bipm1,bipm2})
as a proper formalization of classical Riemann surfaces and infinite ramification points as mathematicians of the XIXth century understood them, in particular
Bernhard Riemann.

These Riemann surfaces are endowed with distinguished charts and provide a direct link to classical special functions. Log-Riemann surfaces are Riemann domains over $\mathbb{C}$. Lifting the flat Euclidean metric defines the \textit{log-Euclidean metric}, and studying the completion of the associated length space, we can define properly the notion of ramification locus, and in particular of infinite ramification points. The original approach from \cite{bipmarxiv} is by explicit construction of the canonical chart by ``cut and paste'' techniques. Then we obtain a Riemann surface with a local diffeomorphism $\pi\colon \mathcal{S} \to \mathbb{C}$. Conversely, as presented in \cite{bipm1}, we can start with $\pi$ and define log-Riemann surfaces. The set of points $\mathcal{R}$ added in the completion ${\mathcal{S}}^* = \mathcal{S} \sqcup \mathcal{R}$ of $\mathcal{S}$ for the log-Euclidean metric on $\mathcal{S}$ is the ramification locus $\mathcal{R}$. Points in $\mathcal{R}$ are at finite distance and the completion ${\mathcal{S}}^*$ is a complete metric space, but is no longer a surface in general, it may not even be a locally compact space.

Isolated points in $\mathcal{R}$ are called \textit{ramification points}.
We only consider in this article the case
where this ramification locus is discrete.
Then the local inverse of $\pi$ composed with a local chart
is fluent in the sense of Ritt (see~\cite{Ritt1} and the forthcoming Ph.D.~thesis by Y.~Levagnini).
Also in this case, the mapping $\pi$ extends continuously to the ramification points $p \in \mathcal{R}$, and is a covering
of a punctured neighborhood of $p$ onto a punctured disk in $\mathbb{C}$.
The point $p$ is a ramification point of $\mathcal{S}$ and its \textit{order} is
equal to the degree of the covering~$\pi$ near~$p$.
The finite order ramification points may be added to $\mathcal{S}$ and give a Riemann
surface ${\mathcal{S}}^\times$, called the finite completion of $\mathcal{S}$. When the number of ramification points
(finite or infinite order) is finite, and the fundamental group is finitely generated we talk about \textit{transalgebraic curves}
that is a~generalization of classical algebraic curves allowing infinite ramification points.

Our goal is to develop an algebraic theory of the function spaces on this transalgebraic curves as is
classically done with algebraic curves. Algebraic functions, and the field of meromorphic functions,
form the backbone of the classical theory of R.~Dedekind and H.~Weber (that will be referred as Dedekind--Weber theory),
originally developed in~\cite{DW}, that represents the historical precursor of the modern commutative algebra and algebraic geometry approach.
For transalgebraic curves, the base function spaces are formed by transcendental functions as we will see in this article.

We study this problem in the simplest situation of genus $0$, i.e., we assume that ${\mathcal{S}}^\times$ is simply connected. Then
${\mathcal{S}}^\times$ is parabolic and biholomorphic to $\mathbb{C}$ (see \cite{bipm2,bipmarxiv}).
Also we proved there (see also the early work by R. Nevanlinna \cite{nevanlinna1,nevanlinna2} and M.~Taniguchi \cite{taniguchi1,taniguchi2}) that we have an explicit
formula for the uniformization $ \tilde{F}\colon \mathbb{C} \to {\mathcal{S}}^\times$
that is given by an entire function $F = \pi \circ \tilde{F}$ of the form
\begin{gather}\label{eq:1}
F(z) = \int Q(z) {\rm e}^{P(z)} \,{\rm d}z,
\end{gather}
where \looseness=-1 $P$ and $Q$ are polynomials of respective degrees $d_1$ and $d_2$, where~$d_1$, resp.~$d_2$, is the number of infinite order, resp.\ finite order, ramification points. Conversely, given $P, Q \in \mathbb{C}[z]$ polynomials of degrees $d_1$, $d_2$ and~$F$ an entire function of the form~\eqref{eq:1}
there exists a~log-Riemann surface~$\mathcal{S}$ with $d_1$ infinite order ramification points and~$d_2$ finite order ramification points (counted with
multiplicity) such that~$F$ lifts to a biholomorphism $\tilde{F}\colon \mathbb{C} \to {\mathcal{S}}^\times$. This can be proved by seeing~$F$ appear as a
limit of Schwarz--Christoffel uniformizations (see~\cite{bipm2} and \cite[Section~II.5.4]{bipmarxiv}).

We limit our study to the simpler situation with no finite order ramification points, so $d_2=0$ and $Q=1$, and we denote $d=d_1\geq 1$.
For $k\geq 0$, we consider the functions
\[
F_k(z)= \int_0^z t^k{\rm e}^{P_0(t)} \,{\rm d}t
\]
and in particular $F_0$ whose lift is the uniformization of the log-Riemann surface under consideration.
The $\mathbb{C}$-vector space $\mathbb{V}_{P_0}$ of transcendental functions
\[
F(z)= \int_{z_0}^z Q(t) {\rm e}^{P_0(t)} \, {\rm d}t,
\]
where $Q\in \mathbb{C}[z]$ and $z_0\in \mathbb{C}$ plays the same role for the associated $\mathcal{S} = {\mathcal{S}}_{P_0}$ than the vector space of polynomials
 $\mathbb{C}[z]$ for the complex plane~$\mathbb{C}$. It was proved in~\cite[Section~III.3]{bipmarxiv} that the functions in this vector space can be characterized by their growth at infinite (``infinite'' in $\mathcal{S}$ being understood as its Alexandrov one-point compactification) by a Liouville type theorem that we recall in Section~\ref{subsection:Liouville}. Without any reference to log-Riemann surface theory, we further study
in Section~\ref{section1} by elementary methods a transcendental base for the ring generated by these functions.
We follow the classical path traced by N.H.~Abel and other mathematicians of the XIXth century to search for a minimal base of transcendentals in order to compute
all these integrals, as was done in Abel's study of Abelian integrals (for the historical development of this ideas one can consult Chapter~IX of~\cite{Lutzen}).
We show that the $d$ transcendentals $F_0, \dots , F_{d-1}$ are algebraically independent and are sufficient to compute the remaining integrals. These functions define Picard--Vessiot extensions of Liouville type of $\mathbb{C}(z)$. We also study the Liouville classification of these transcendentals from the old pre-differential algebra
Liouville classification (see \cite{Liou1,Liou2}).

After these preliminaries in Section~\ref{section1}, we turn to study the asymptotic values of $F_0, \dots , F_{d-1}$, that are
transcendental exponential periods (as defined by D.~Zagier and M.~Kontsevich in~\cite{KZ2001})
\[
\Omega_{kl}(P_0) =\int_0^{+\infty.\omega_l} t^{k-1} {\rm e}^{P_0(t)} \, {\rm d}t
\]
for \looseness=-1 $k=1,\dots, d$, where we normalize $P_0(t)=-\frac1d t^d+\cdots $ and $(\omega_l)_{1\leq l\leq d}$ are the $d$-th
roots of unity. These periods are in general non-computable integrals. We define the \textit{Ramificant determinant} by
\[
\Delta(P_0)=\left |
\begin{matrix}
\Omega_{11} & \Omega_{12} & \dots &\Omega_{1d} \\
\Omega_{21} & \Omega_{22} & \dots &\Omega_{2d} \\
\vdots &\vdots &\ddots & \vdots \\
\Omega_{d1} & \Omega_{d2} & \dots &\Omega_{dd} \\
\end{matrix}
\right | .
\]

Even if the $\Omega_{kl}$'s are non-computable, one of the fundamental results established in Section~\ref{section2} is
that the Ramificant determinant is computable, and we give a closed-form formula:

\begin{Theorem} \label{detformula}For $d\geq 1$, there exists $\Pi_d$, a universal polynomial with rational coefficients on the coefficients of $P_0$, such that
\[
\Delta(P_0)=\frac{( 2\pi d )^{\frac{d}{2}}}{\sqrt {2\pi}} \exp({\Pi_d }).
\]
\end{Theorem}

In particular we get the trivial, but fundamental, corollary that the Ramificant determinant never vanishes, $\Delta(P_0)\not= 0$.
From this non-vanishing result, we obtain in Section~\ref{section3} an Abel-like theorem, that can be seen as a criterion for integrability in finite terms \`a~la Abel and Liouville. Also it follows a Torelli-like theorem that proves that the periods determine the polynomial~$P_0$. These results were extended by the first author to finite type log-Riemann surfaces (see~\mbox{\cite{bibi,biderham}}). Another corollary is that the period mapping is \'etale, and a transalgebraic version of fundamental symmetric formulas. In Section~\ref{section4} we develop applications to the transalgebraic theory of log-Riemann surfaces. To $\mathbb{V}_{P_0}$ it corresponds the vector space of functions on $\mathcal{S}$, $\mathcal{V}_{\mathcal{S}}$, that generates the structural ring $\hat{\mathcal{A}}_{\mathcal{S}}$. We prove that this ring of functions separates points on the log-Riemann surface $\mathcal{S}$, including the infinite ramification points in the completion. The transcendental functions in the structural ring have Stolz limits at the infinite ramification points, thus the algebraic theory extends to these points also (a Stolz limit corresponds to a limit through an angular sector and this is an important notion in the theory of conformal representation, see \cite[p.~6]{Pommerenke}) The points of $\mathcal{S}^*$ are identified with some maximal ideals of the structural ring. We also explain how to distinguish algebraically the finite ramification points from the infinite ones.

Functions in the vector space $\mathcal{V}_{\mathcal{S}}$ can be characterized by their growth at infinite, i.e., by an extension of the classical Liouville theorem to this setting.

Most of the results presented in this article are collected from the algebraic part (Section~III) of the original manuscript \cite{bipmarxiv} that dates back to 2003--2005.

\section{A ring of special functions}\label{section1}

\subsection{Definitions}

Let $P_0(z)\in \mathbb{C}[z]$ be a polynomial of degree $d\geq 1$
\[
P_0(z)=a_d z^d+a_{d-1} z^{d-1}+\dots +a_1 z+a_0.
\]
We consider the entire functions
\begin{gather*}
F_0(z)=\int_0^z {\rm e}^{P_0(t)} \,{\rm d}t, \qquad F_1(z)=\int_0^z t\ {\rm e}^{P_0(t)} \,{\rm d}t, \qquad
\dots,\qquad F_{d-1}(z)=\int_0^z t^{d-1} {\rm e}^{P_0(t)} \,{\rm d}t.
\end{gather*}
and the $\mathbb{C}$-vector space generated by these transcendental functions and constant functions
\[
\mathbb{U}_{P_0} =\langle \mathbb{C}, F_0,\dots , F_{d-1}\rangle.
\]

\begin{Proposition}\label{prop:exp_linear_combo}
We have
\[
{\rm e}^{P_0} \in \mathbb{U}_{P_0}.
\]
\end{Proposition}

\begin{proof}Since
\[
{\rm e}^{P_0(z)}-{\rm e}^{P_0(0)}=\int_0^z P_0'(t) {\rm e}^{P_0(t)} \,{\rm d}t
\]
we get
\begin{gather*}
{\rm e}^{P_0}={\rm e}^{P_0(0)} \cdot 1+a_1 F_0+2a_2 F_1+\cdots + (d-1)a_{d-1} F_{d-2}+ da_d F_{d-1} . \tag*{\qed}
\end{gather*}
\renewcommand{\qed}{}
\end{proof}

We prove in the next sections that $1, F_0, \dots, F_{d-1}$ are $\mathbb{C}$-linearly independent and also algebraically independent.

\begin{Definition}
We consider the ring generated by polynomials $\mathbb{C}[z]$ adjoining $F_0,\dots , F_{d-1},$
\[
\mathbb{A}_{P_0}=\mathbb{C}[z] [F_0, \dots, F_{d-1}].
\]
Let ${\mathbb{K}}_{P_0}$ be the field of fractions of $\mathbb{A}_{P_0}$, thus ${\mathbb{K}}_{P_0}$ is the extension of the field of rational functions~$\mathbb{C}(z)$
adjoining $F_0,\dots , F_{d-1}$,
\[
{\mathbb{K}}_{P_0}=\mathbb{C}(z)(F_0 ,\dots , F_{d-1}) =\mathbb{C} (z, F_0, \dots , F_{d-1}).
\]
\end{Definition}

Our first goal is to prove:

\begin{Theorem}\label{thm:alg_indep}The field ${\mathbb{K}}_{P_0}$ has transcendence degree $d+1$ over $\mathbb{C}$.
\end{Theorem}

\subsection{Asymptotics at infinite}\label{subsection:asymptotics}

The following asymptotic estimate is key in the proofs of the algebraic results.

\begin{Proposition}\label{prop:asymptotics} For $j=0,1,\dots, d-1$ we have
\[
F_j(z)\sim \frac{z^j}{P_0'(z)} {\rm e}^{P_0(z)},
\]
when $z\to +\infty . a_d^{-1/d}$, that is when $z\to \infty$ in a direction given by a $d$-root of $a_d^{-1}$.
\end{Proposition}

\begin{proof}In these directions $P_0$ and $P'_0$ tends to $+\infty$, thus we can assume that $P'_0$ is non zero at $0$ by changing the
origin of integration (i.e., by a translation change of variables in the integrals).
Performing two integration by parts we get
\begin{align*}
F_j(z) &=\int_0^z t^j {\rm e}^{P_0(t)} \,{\rm d}t =\int_0^z \frac{t^j}{P'_0(t)} P'_0(t){\rm e}^{P_0(t)} \,{\rm d}t \\
&=\left [ \frac{t^j}{P'_0(t)} {\rm e}^{P_0(t)} \right ]_0^z -\int_0^z \left (\frac{jt^{j-1} P'_0(t) -t^j P''_0(t)}{ (P'_0(t) )^2}\right ) {\rm e}^{P_0(t)}\,{\rm d}t\\
&=\frac{z^j}{P'_0(z)} {\rm e}^{P_0(z)}-\int_0^z \mathcal{O} \left(\frac{1}{\big(a_d^{1/d} t\big)^{d-j}}\right ) {\rm e}^{P_0(t)} \,{\rm d}t\\
&=\frac{z^j}{P'_0(z)}{\rm e}^{P_0(z)}-\left [\mathcal{O} \left (\frac{1}{\big(a_d^{1/d} t\big)^{d-j}}\right )
\frac{1}{P'_0(t)} {\rm e}^{P_0(t)}\right ]_0^z +\int_0^z \mathcal{O} \left (\frac{1}{
\big(a_d^{1/d} t\big)^{2d-j-2}}\right ) {\rm e}^{P_0(t)}\,{\rm d}t.
\end{align*}
Now the two last terms in the last equation are dominated by the first one.
\end{proof}

\subsection{Linear independence}

\begin{Proposition}\label{prop:linear_indep}The constant function $1$ and the special functions $F_0, F_1, \dots , F_{d-1}$ are linearly independent over $\mathbb{C}$.
\end{Proposition}

We give different proofs of this Proposition.

\begin{proof}[1st proof]
Consider a non-trivial linear combination
\[
b_{-1}+b_0 F_0+b_1 F_1+\dots +b_{d-1} F_{d-1}=0,
\]
and take one derivative. Dividing by ${\rm e}^{P_0}$ we get
\[
b_0+b_1 z+\dots +b_{d-1} z^{d-1}=0.
\]
Thus we get $b_0=b_1=\dots =0$ and then $b_{-1}=0$ also.
\end{proof}

Now we give an analytic proof.

\begin{proof}[2nd proof] Consider a non-trivial linear combination
\[
b_{-1}+b_0 F_0+b_1 F_1+\cdots +b_{d-1} F_{d-1}=0
\]
and let $0\leq k\leq d-1$ be the largest index such that $b_k\not=0$. If $k=-1$ we are done. If not, when $z\to + \infty . a_d^{-1/d}$ we have
\[
b_{-1}+b_0 F_0+b_1 F_1+\dots +b_{d-1} F_{d-1} \sim b_k \frac{z^k}{P'_0(z)} {\rm e}^{P_0(z)}\to \infty.
\]
We have a contradiction.
\end{proof}

Finally we give a more algebraic proof.

\begin{proof}[3rd proof]First we show that $F_0, \dots , F_{d-1}$ are $\mathbb{C}$-linearly independent. Choose $d$ distinct points $z_0, z_1, \dots , z_{d-1} \in \mathbb{C}$. If a linear combination $b_0 F_0 +b_1 F_1+\dots +b_{d-1} F_{d-1}$ vanishes at $z_0, z_1, \dots , z_{d-1} \in \mathbb{C}$ then we have
\[
\Delta (z_0, \dots , z_{d-1})=\left |
\begin{matrix}
F_0(z_0) & F_0 (z_1) & \dots & F_0(z_{d-1}) \\
F_1(z_0) & F_1(z_1)& \dots & F_1(z_{d-1})\\
\vdots & \vdots & \ddots &\vdots \\
F_{d-1}(z_0) & F_{d-1}(z_1) &\dots &F_{d-1}(z_{d-1})\\
\end{matrix}
\right | =0.
\]
But we have
\[
\partial_{z_{d-1}} \cdots \partial_{z_1} \partial_{z_0} \Delta =
{\rm e}^{P_0(z_0)} . {\rm e}^{P_0(z_1)} \cdots {\rm e}^{P_0(z_{d-1})}.
\left |
\begin{matrix}
1 & 1 & \dots & 1 \\
z_0 & z_1& \dots & z_{d-1}\\
\vdots & \vdots & \ddots &\vdots \\
z_0^{d-1} & z_1^{d-1} &\dots &z_{d-1}^{d-1}\\
\end{matrix}
\right |
\]
and the Vandermonde determinant is not zero, thus $\partial_{z_{d-1}} \cdots \partial_{z_1} \partial_{z_0} \Delta \not= 0$ and $\Delta$ is not identical\-ly~$0$. Contradiction.

In order to show that $1, F_0, \dots , F_{d-1}$ are $\mathbb{C}$-linearly independent we proceed in a similar way evaluating the linear combination at $d$ distinct points $z_0, z_1, \dots , z_{d-1}$. We consider the same determinant $\Delta$ adding a first column and a first row of ones. Next we apply the differential operator $\partial_{z_1, z_2, \dots , z_{d-1}}$ to $\Delta$ and develop the resulting determinant through the first row and we get a contradiction as before.
\end{proof}

We can now prove more.
\begin{Proposition}\label{prop:linear_indep2}The special functions $F_0, F_1, \dots , F_{d-1}$ and the constant function $1$ are linearly independent over the ring of polynomials $\mathbb{C} [z]$.
\end{Proposition}

\begin{proof}By contradiction consider a non-trivial linear combination with polynomial coefficients
\begin{gather*}
A_{-1}(z)+A_0(z)F_0(z)+\dots +A_{d-1}(z) F_{d-1}(z) =0.
\end{gather*}
Taking one derivative we get
\[
A'_{-1}(z)+A'_0(z)F_0(z)+\dots +A'_{d-1}(z) F_{d-1}(z)=Q_1(z){\rm e}^{P_0(z)},
\]
where $Q_1(z)=-A_0(z)-zA_1(z)-\dots -z^{d-1} A_{d-1}(z)$. Iterating this procedure and taking $k$ derivatives, we get
\[
A_{-1}^{(k)}(z)+A^{(k)}_0(z)F_0(z)+\dots +A^{(k)}_{d-1}(z) F_{d-1}(z)=Q_k(z){\rm e}^{P_0(z)},
\]
where $Q_k(z)\in \mathbb{C} [z]$. Choose $k\geq 0$ minimal such that all $A^{(k)}_j$ are constant but not all $0$. Let $-1\leq l_0 \leq d-1$ be the largest index such that $A_{l}^{(k)}\not= 0$. If $l_0=-1$, we have
\[
A_{-1}^{(k)}+A_0^{(k)}F_0(z)+\dots +A_{d-1}^{(k)} F_{d-1}(z)= A_{-1}^{(k)} = Q_k(z){\rm e}^{P_0(z)},
\]
so $Q_k=0$ and $A_{-1}^{(k)}=0$.

If $l_0\geq 0$, then when $z\to +\infty . a_d^{-1/d}$, using Proposition \ref{prop:asymptotics}, we have the asymptotics
\[
A_{-1}^{(k)}+A_0^{(k)}F_0(z)+\dots +A_{d-1}^{(k)} F_{d-1}(z)\sim A_{l_0}^{(k)} \frac{z^{l_0}}{P'_0(z)} {\rm e}^{P_0(z)}.
\]
But since $l_0 \leq d-1$,
\[
Q_k(z) \sim A_{l_0}^{(k)} \frac{z^{l_0}}{P'_0(z)}
\]
is only possible when $l_0=d-1$. Thus $l_0=d-1$, and the degree of $A_j$ is at most the degree of~$A_{d-1}$. When $z\to +\infty . a_d^{-1/d}$ we have that $A_{d-1} F_{d-1}$ dominates $A_j F_j$ for $j< d-1$. Thus if $c$ is the leading coefficient of $A_{d-1}(z)$ and $m$ is its degree then, when $z\to +\infty . a_d^{-1/d}$, we have
\[
A_{-1}(z)+ A_0(z)F_0(z)+\cdots +A_{d-1}(z) F_{d-1}(z) \sim c \frac{z^{m+d-1}}{P'_0(z)} {\rm e}^{P_0(z)}.
\]
On the other hand $A_{-1}+A_0F_0+\dots +A_{d-1} F_{d-1}= 0$, so $c$ must be $0$, $A_{d-1}$ is zero, as well as all the other $A_j$. We have a contradiction.
\end{proof}

\subsection{Algebraic independence}

We prove now Theorem \ref{thm:alg_indep}, i.e., that the field ${\mathbb{K}}_{P_0}$ has transcendence degree $d$ over $\mathbb{C}(z)$. Clearly the transcendence degree is at most~$d$. That it is exactly $d$ follows from the next result:

\begin{Lemma}\label{lemma:F_k_transcendental} For $k=1,\dots , d-1$, $F_k$ is transcendental over $\mathbb{C} (z, F_0, \dots , F_{k-1})$, and $F_0$ is transcendental over $\mathbb{C}(z)$.
\end{Lemma}

Before proving the lemma, we give a definition.

\begin{Definition}The exponential degree, resp.\ the polynomial degree, of a monomial expression
\[
z^m F_0^{n_0} F_1^{n_1} \cdots F_{d-1}^{n_{d-1}}
\]
are $| \mathbf{n} |=n_0+n_1+\dots +n_{d-1}$, resp.\ $m+n_1+2n_2+\dots +(d-1) n_{d-1}=m+(\mathbf{d}-\mathbf{1}).\mathbf{n}$, where $(\mathbf{d}-\mathbf{1})$ denotes the vector $(0,1,2,\dots , d-1)$, and $\mathbf{n}$ the vector $(n_0, n_1, \dots , n_{d-1})$.
\end{Definition}
\begin{Lemma}In a vanishing $\mathbb{C}$-linear combination of monomials in $z, F_0, \dots, F_{d-1}$ each sub-linear combination of monomials with the same exponential and polynomial degree must vanish.
\end{Lemma}

\begin{proof}We have the asymptotics when $z\to +\infty .a_d^{-1/d}$,
\begin{align*}
z^m F_0^{n_0} F_1^{n_1} \cdots F_{d-1}^{n_{d-1}} & \sim \frac{z^{m+n_1+2n_2
+\dots +(d-1)n_{d-1} }}{(P'_0(z))^{n_0+n_1+\dots +n_{d-1}} }
{\rm e}^{(n_0+n_1+\dots +n_{d-1})P_0(z)} \\
& \sim z^{m+(\mathbf{d-1}).\mathbf{n}-|\mathbf{n} |(d-1)} {\rm e}^{|\mathbf{n} |.P_0(z)}.
\end{align*}

Now consider a vanishing $\mathbb{C}$-linear combination of monomials
\[
0=\sum_{m, \mathbf{n}} a_{m, \mathbf{n}} z^m F_0^{n_0} F_1^{n_1} \cdots F_{d-1}^{n_{d-1}} =
\sum_{N\geq 0} \sum_{m, \mathbf{n} \atop |\mathbf{n}|=N} a_{m, \mathbf{n} } z^m F_0^{n_0} F_1^{n_1} \cdots F_{d-1}^{n_{d-1}}.
\]
The different exponential asymptotics show that for each $N\geq 0$,
\[
0=\sum_{m, \mathbf{n} \atop |\mathbf{n} |=N} a_{m, \mathbf{n}} z^m F_0^{n_0} F_1^{n_1} \cdots F_{d-1}^{n_{d-1}}
=\sum_{m\geq 0} \sum_{\mathbf{n} \atop |\mathbf{n}|=N} a_{m, \mathbf{n}} z^m F_0^{n_0} F_1^{n_1} \cdots F_{d-1}^{n_{d-1}}.
\]
Again the same argument using the different asymptotics for monomials
with the same exponential degree but different polynomial degree gives the
result, that is, for each $N\geq 0$ and $m\geq 0$,
\begin{gather*}
\sum_{\mathbf{n} \atop |\mathbf{n} |=N} a_{m,\mathbf{n}} z^m F_0^{n_0} F_1^{n_1} \cdots F_{d-1}^{n_{d-1}}=0. \tag*{\qed}
\end{gather*}
\renewcommand{\qed}{}
\end{proof}

\begin{Lemma}\label{lemma:C[x]-indep}
Let $N\geq 1$. The monomials $F_0^{n_0} F_1^{n_1}\cdots F_k^{n_k}$ of exponential degree $N$ are $\mathbb{C}[z]$-linearly independent.
\end{Lemma}

\begin{proof}We prove the result by induction on $N\geq 1$. For $N=1$ we have the result by Proposition~\ref{prop:linear_indep2}. Assume the result for $N-1$ and consider, by
contradiction, a non-trivial $\mathbb{C}[z]$ linear dependence relation
\[
\sum_{\mathbf{n}} A_{\mathbf{n}} (z) F_0^{n_0} F_1^{n_1} \cdots F_k^{n_k}=0.
\]
We can assume using the previous lemma that each term in this sum has the same polynomial degree (we could also assume for the same reasons that each polynomial $A_n(z)$ is a monomial, but we don't need that). This means that there exists a constant $K$ such that for each $\mathbf{n}$
\[
\deg A_{\mathbf{n}}+{\mathbf{k}}.{\mathbf{n}}=K,
\]
where $\mathbf{k} =(0,1,2,\dots , k)$. Taking one more derivative to the precedent relation we get
\[
\sum_{\mathbf{n}} A'_{\mathbf{n}} (z) \ F_0^{n_0} F_1^{n_1} \cdots F_k^{n_k}
=-\sum_{{\mathbf{n}} \atop j=0,1,\dots, k} z^j A_{\mathbf{n}}(z)
F_0^{n_0} \cdots F_j^{n_j-1}\cdots F_k^{n_k} {\rm e}^{P_0}.
\]
Note that the exponential degree of the terms on the right hand side is the same as the one on the left side, but the polynomial degrees are greater by $1$, therefore
\[
\sum_{\mathbf{n}} A'_{\mathbf{n}} (z) F_0^{n_0} F_1^{n_1} \cdots F_k^{n_k} =0.
\]
We continue taking derivatives and stop one step before all $A^{(l+1)}_{\mathbf{n}}$ vanish, that is when
\[
\sum_{\mathbf{n}} A^{(l)}_{\mathbf{n}} F_0^{n_0} F_1^{n_1} \cdots F_k^{n_k} =0,
\]
is a non-trivial $\mathbb{C}$-linear combination of homogeneous monomials on the $F_j$'s. Observe now that taking one more derivative in this last relation and dividing by ${\rm e}^{P_0}$ gives
\[
\sum_{\mathbf{n} \atop j=0,1,\dots , k} A^{(l)}_{\mathbf{n}} z^j F_0^{n_0} \cdots F_{j-1}^{n_{j-1}} F_j^{n_j-1} F_{j+1}^{n_{j+1}} \cdots F_k^{n_k} =0.
\]
Observe that each monomial in $z, F_0, \dots , F_k$ in this sum comes from exactly one monomial in $F_0, \dots , F_k$ of the relation we have differentiated. And this last relation is a non-trivial $\mathbb{C}[z]$-linear combination between monomials of exponential degree~$N-1$. By induction assumption this is impossible.
\end{proof}

\begin{proof}[Proof of Theorem \ref{thm:alg_indep}] It is enough to prove Lemma~\ref{lemma:F_k_transcendental}. If $F_k$ is not transcendental over $\mathbb{C}(z, F_0, \dots , F_{k-1})$, then we have a non-trivial polynomial relation between $z, F_0, \dots, F_k$. Iso\-lating parts of the same exponential degree we are lead to a~non-trivial $\mathbb{C}[z]$-linear relation between homogeneous monomials in $F_0, \dots , F_k$ which contradicts the previous Lemma~\ref{lemma:C[x]-indep}.
\end{proof}

\subsection{Computation of integrals}

We adopt here a similar point of view to Abel and his contemporaries on elliptic functions and, in general, Abelian integrals. The special functions $F_0, F_1, \dots, F_{d-1}$ are all we need in order to compute a large class of integrals, or ``transcendentals'' as Abel would put it. As for Abelian integrals, next theorem shows that computable integrals have finite codimension in the family of integrals considered.

\begin{Theorem}\label{thm:computation_integrals}
We consider the $\mathbb{C}$-vector space
\begin{align*}
\mathbb{V}_{P_0}=\mathbb{V}_{P_0}(\mathbb{C}) &=\mathbb{C}[z] . {\rm e}^{P_0(z)} \oplus
\mathbb{C} . F_0 \oplus \dots \oplus \mathbb{C} . F_{d-2} \\
&=z\mathbb{C}[z] . {\rm e}^{P_0(z)} \oplus \mathbb{C} . 1 \oplus \mathbb{C} . F_0 \oplus \dots \oplus \mathbb{C} . F_{d-1}.
\end{align*}
For $Q(z)\in \mathbb{C}[z]$, any primitive
\[
\int_0^z Q(t) {\rm e}^{P_0(t)} \, {\rm d}t
\]
is in the vector space $\mathbb{V}_{P_0}$. Conversely, any point of the hyperplane of $\mathbb{V}_{P_0}$ of functions vanishing at $0$ is such a primitive
\[
\left\{ \int_0^z Q(t) {\rm e}^{P_0(t)} \, {\rm d}t;\, Q(z)\in \mathbb{C}[z] \right \} = \{ F\in \mathbb{V}_{P_0} ;\, F(0)=0\}.
\]
We have
\[
\mathbb{V}_{P_0}=\left \{\int_{z_0}^z Q(t) {\rm e}^{P_0(t)} \, {\rm d}t; \, z_0\in \mathbb{C} ,\, Q(z)\in \mathbb{C}[z] \right \}.
\]
\end{Theorem}

\begin{proof}First note that the equality of the two sums results from the fact that ${\rm e}^{P_0}$ is a $\mathbb{C}$-linear combination of $F_0, \dots , F_{d-1}$, and the direct sums result from the algebraic independence proved in the previous section. We prove the result by induction on the degree of~$Q$. The result is clear for $\deg Q \leq d-2$ because then $\int Q {\rm e}^{P_0}\, {\rm d}t$ is a linear combination of $1, F_0, \dots , F_{d-2}$. For $\deg Q\geq d-1$, we take the Euclidean division of $Q$ by $P'_0$,
\[
Q=AP'_0 +B,
\]
where $A,B\in \mathbb{C}[z]$ and $\deg B <d-1$.
Then, by integration by parts it follows
\begin{align*}
\int_0^z Q(t) {\rm e}^{P_0(t)} \, {\rm d}t &= \int_0^z (A(t)P'_0(t) +B(t)) {\rm e}^{P_0(t)} \, {\rm d}t \\
&=\left [A(t){\rm e}^{P_0(t)}\right ]_0^z -\int_0^z A'(t) {\rm e}^{P_0(t)} \, {\rm d}t
+\int_0^z B(t) {\rm e}^{P_0(t)} \, {\rm d}t \\
&=A(z){\rm e}^{P_0(z)} -A(0){\rm e}^{P_0(0)} -\int_0^z A'(t) {\rm e}^{P_0(t)} \, {\rm d}t
+\int_0^z B(t) {\rm e}^{P_0(t)} \, {\rm d}t.
\end{align*}

Now we have $A(z){\rm e}^{P_0(z)} \in \mathbb{C}[z] {\rm e}^{P_0(z)}$, $-A(0){\rm e}^{P_0(0)} \in \mathbb{C}$, and the primitive $\int_0^z B(t) {\rm e}^{P_0(t)} \, {\rm d}t$ is a~linear combination of $F_0, \dots, F_{d-2}$. Moreover, we have $\deg A' < \deg Q$ so the result follows by induction.

For the converse, let $F\in \mathbb{V}_{P_0}$ vanishing at $0$ and write
\[
F(z)=zP(z){\rm e}^{P_0(z)} +c_0+c_1 F_0+\dots+ c_{d} F_{d-1},
\]
where $P(z)\in \mathbb{C}[z]$ and $c_0, c_1, \dots, c_{d} \in \mathbb{C}$. Since $F(0)=0$ we have $c_0=0$. Also
\[
c_1 F_0+\dots c_{d} F_{d-1} =\int_0^z \big(c_1 +c_2 t+\dots +c_{d} t^{d-1} \big) {\rm e}^{P_0(t)} \, {\rm d}t,
\]
and
\begin{gather*}
zP(z) {\rm e}^{P_0(z)} =\int_0^z \big(P(t)+tP'(t)+tP(t)P'_0(t) \big) {\rm e}^{P_0(t)} \, {\rm d}t. \tag*{\qed}
\end{gather*}
\renewcommand{\qed}{}
\end{proof}

\begin{Remark}\textbf{1.} Let ${\mathbb{K}} \subset \mathbb{C}$ be a subfield of the complex numbers. If $P_0(z)\in \mathbb{K}[z]$ and $P_0$ is normalized such that $P_0(0)=0$, then any primitive
\[
\int_0^z Q(t) {\rm e}^{P_0(t)} \, {\rm d}t,
\]
where $Q(z)\in \mathbb{K}[z]$ belongs to the $\mathbb{K}$-vector space
\[
\mathbb{V}_{P_0}(\mathbb{K})=z\mathbb{K}[z] {\rm e}^{P_0(z)} \oplus \mathbb{K} \oplus \mathbb{K} F_0 \oplus \dots \oplus \mathbb{K} F_{d-1}.
\]
This results from the previous proof since the Euclidean division of polynomials is well defined in the ring $\mathbb{K} [z]$, and ${\rm e}^{P_0(0)}=1$. The proof of the converse statement is analogous.

{\bf 2.} In general, let $\mathbb{K}$ be a field and consider the differential ring $\mathbb{K}[z]$. For $P_0\in \mathbb{K}[z]$, $\deg P_0 =d$, we define ${\rm e}^{P_0}$ as generating the Liouville extension defined by the differential equation
\[
y'-P_0 y=0.
\]
We consider the extension $\mathbb{K}_0$ generated by
\begin{gather*}
y' = {\rm e}^{P_0}, \qquad y' = z {\rm e}^{P_0},\qquad \dots, \qquad y' = z^{d-1} {\rm e}^{P_0},
\end{gather*}
and denote by $F_0, F_1, \dots , F_{d-1}$ these primitives. Then the $\mathbb{K}$-vector space
\[
\mathbb{M}_{P_0}=z\mathbb{K}[z] {\rm e}^{P_0} \oplus \mathbb{K} . 1\oplus \mathbb{K} . F_0 \oplus \dots
\oplus \mathbb{K} . F_{d-1}
\]
coincides with the set of all primitives $\int Q {\rm e}^{P_0}\,{\rm d}t$ modulo constants.
\end{Remark}

\subsection{Differential ring structure}

We denote by $D=\frac{{\rm d}}{{\rm d}z}$ the differentiation operator in the ring $\mathbb{A}_{P_0}$. Let $\mathbb{A}^{N,n}_{P_0}$ be the $\mathbb{C}$-module generated by those monomials of exponential degree $N$ and polynomial degree $n$. We have the graduation
\[
\mathbb{A}_{P_0}=\bigoplus_{N,n\geq 0} \mathbb{A}_{P_0}^{N,n}.
\]
The following proposition is immediate.

\begin{Proposition}We have
\[
D \mathbb{A}_{P_0}^{N,n} \subset \mathbb{A}_{P_0}^{N,n-1} \oplus \big(\mathbb{A}_{P_0}^{N-1,n} \oplus
\mathbb{A}_{P_0}^{N-1,n+1}\oplus \dots \oplus \mathbb{A}_{P_0}^{N-1,n+d-1}\big) {\rm e}^{P_0}.
\]
In particular, the principal ideal $\big({\rm e}^{P_0}\big)$ generated by ${\rm e}^{P_0}$ is absorbing for the derivation, i.e., any element of $\mathbb{A}_{P_0}$ ends up into $\big({\rm e}^{P_0}\big)$ after a finite number of derivatives.
\end{Proposition}

Next we determine the elements of $\mathbb{A}_{P_0}$ without zeros.

\begin{Proposition}\label{prop:charac_exp}The only elements in $\mathbb{A}_{P_0}$ without zeros are
\[
\mathbb{C}^* \cup \big\{ {\rm e}^{nP_0} ;\, n\geq 1\big\},
\]
that is, the non-zero constant functions and ${\rm e}^{P_0}, {\rm e}^{2P_0}, \dots$.

The group of units in $\mathbb{A}_{P_0}$ is composed by the non-vanishing constant functions
\[
\mathbb{A}_{P_0}^{\times} =\mathbb{C}^*.
\]
\end{Proposition}

\begin{proof}Let $F\in \mathbb{A}_{P_0}$ without zeros. Since $\mathbb{A}_{P_0}$ is a ring of entire functions of order at most $d$, and $F$ is zero free, we can find a polynomial of degree $\leq d$ such that
\[
F=e^Q.
\]
Now, when $z\to +\infty . a_d^{-1/d}$, using Proposition \ref{prop:asymptotics}, the asymptotics of each $F\in \mathbb{A}_{P_0}$ is of the form
\[
F(z)\sim c z^{a}{\rm e}^{bP_0(z)},
\]
where $c\in \mathbb{C}$, and $a,b \in \mathbb{N}$, $b\geq 0$. Therefore we must have $Q=nP_0$ for some $n\geq 1$ or $Q$ is a~constant polynomial (case $b=0$). This proves the first statement.

For the second statement, let $F\in \mathbb{A}_{P_0}^{\times}$ be invertible. Then $1/F$ belongs to the ring, so it is holomorphic. Thus $F$ has no zeros. Moreover $F$ cannot be of the form ${\rm e}^{nP_0}$ for $n\geq 0$ since
\[
{\rm e}^{-nP_0(z)} \to 0,
\]
when $z\to +\infty . a_d^{-1/d}$ and we know that for any non-constant element $G$ in the ring $\mathbb{A}_{P_0}$
\[
G(z)\to +\infty,
\]
when $z\to +\infty . a_d^{-1/d}$.
\end{proof}

\subsection{Picard--Vessiot extensions}

We recall that a Picard--Vessiot extension of a differential ring $\mathbb{A}$ is a differential ring extension $\mathbb{A}[u_1, \dots , u_n]$ generated by $u_1, \dots , u_n$ fundamental solutions of an homogeneous linear differential
equation of order $n$
\[
y^{(n)} +b_{n-1} y^{(n-1)}+\dots +b_1 y'+b_0 y=0,
\]
where $b_j \in \mathbb{A}$ and the ring of constants of the extension coincides with the ring of constants of~$\mathbb{A}$.

We recall also that a {\it Liouville extension} is a~Picard--Vessiot extension generated by successive adjunctions of integrals or exponentials of integrals (see \cite[Chapter~III.12, p.~23]{Ka} and~\cite{Ritt2}). These have a solvable differential Galois group \cite[Chapter~III.13, p.~24]{Ka}.

\begin{Theorem}The field $\mathbb{K}_{P_0}=\mathbb{C} (z, F_0, \dots, F_{d-1})$ and the ring $\mathbb{A}_{P_0}=\mathbb{C} [z, F_0, \dots, F_{d-1}]$ are Picard--Vessiot extensions of $\mathbb{C}(z)$ and $\mathbb{C}[z]$ respectively, i.e., they are generated by the fundamental solutions of a linear homogeneous differential equation with polynomial coefficients. Moreover these extensions are Liouville extensions.
\end{Theorem}

The ring of constants are the constant functions. We only need to find the homogeneous linear differential equation satisfied by $F_0, \dots , F_{d-1}$. We construct a homogeneous linear differential equation satisfied by $F'_0,\dots , F'_{d-1}$.

We define a double sequence of functions $(y_{n,m})_{n\in \mathbb{Z} \atop m \geq 0}$ by
\begin{itemize}\itemsep=0pt
\item $y_{0,0}={\rm e}^{P_0}$,
\item for $n>m$, $y_{n,m}=0$,
\item for $n<0$, $y_{n,m}=0$,
\item for $n \in \mathbb{N}$, $m\geq 0$,
\begin{gather*}
y_{n,m+1}=y_{n-1,m}+y'_{n,m}
\end{gather*}
(Pascal's triangle rule with one derivative).
\end{itemize}

The first lemma is straightforward.

\begin{Lemma} \label{lemma:more_formulas}
We have
\begin{itemize}\itemsep=0pt
\item for $n\geq 0$, $y_{n,n}={\rm e}^{P_0}$,
\item for $m\geq 0$, $y_{0,m}=\big( {\rm e}^{P_0} \big)^{(m)}$,
\item for all $n\in \mathbb{N}$, $m\geq 0$, $y_{n,m}=Q_{n,m} {\rm e}^{P_0}$, where $Q_{n,m}$ is a universal polynomial with positive integer coefficients on $P'_0, P''_0, P_0^{(3)}, \dots$.
\end{itemize}
\end{Lemma}

And we need a second lemma:

\begin{Lemma}\label{lemma:formulas}We define for $k\geq0$, $y_k(z)= z^k {\rm e}^{P_0(z)}=z^k y_{k,k}$. Then we have
\begin{itemize}\itemsep=0pt
\item for $0\leq l\leq k$,
\[
y_k^{(l)}=z^k y_{0,l}+kz^{k-1} y_{1,l}+k(k-1) z^{k-2} y_{2,l}+\dots +{\frac{k!}{(k-l)!}} z^{k-l} y_{l,l},
\]
\item for $k\leq l$,
\[
y_k^{(l)}=z^k y_{0,l}+kz^{k-1} y_{1,l}+k(k-1) z^{k-2} y_{2,l}+\dots +{\frac{k!}{1}} z y_{k-1,l}+k! y_{k,l}.
\]
\end{itemize}
\end{Lemma}

\begin{proof} It results from a direct induction on $l$ observing that $y'_{0,l}=y_{0,l+1}$ and $y_{0,l}+y'_{1,l}=y_{1,l+1}$, and so on.
\end{proof}

\begin{proof}[Proof of the theorem] We look for polynomials $b_0, b_1, \dots , b_{d-1}$ such that $y_0=F_0', y_1=F_1'$, $\dots , y_{d-1}=F_{d-1}'$ are solutions of
\[
y^{(d)}+b_{d-1} y^{(d-1)}+\dots + b_1 y' +b_0 y =0 .
\]
They will form a fundamental set of solutions since these functions are $\mathbb{C}$-linearly independent. Once we find these polynomial coefficients, the special functions $1,
F_0, F_1, \dots , F_{d-1}$ will form a~fundamental set of solutions of
\[
y^{(d+1)}+b_{d-1} y^{(d)}+\dots +b_1 y'' +b_0 y' =0 .
\]
We can plug $y_k$ into the differential equation and compute $y_k^{(l)}$ using Lemma~\ref{lemma:formulas}.
Then grouping together the factors of $z^j$, $j=0, \dots , d-1$, we get a triangular system
\[
b_j y_{j,j} +b_{j+1} y_{j,j+1} +\dots +b_{d-1} y_{j,d-1} +y_{j,d} =0.
\]
Then, since $y_{j,j}={\rm e}^{P_0}$, we get
\[
b_j=-b_{j+1} y_{j,j+1} {\rm e}^{-P_0}-\dots -b_{d-1} y_{j,d-1} {\rm e}^{-P_0} -y_{j,d} {\rm e}^{-P_0},
\]
and the result follows using Lemma \ref{lemma:more_formulas}.
Note that the extension is a Liouville extension as claimed since each
$F_0$ is the exponential of an integral followed by an integral, and
for $j\geq 1$ the special function $F_j$ is an integral over the field
generated by ${\rm e}^{P_0}$.
\end{proof}

\begin{Remark}The Wronskian of $F_0, F_1, \dots, F_{d-1}$ satisfies the differential equation
\[
W'-d P'_0 W=0 ,
\]
and is equal to $W(z)={\rm e}^{dP_0(z)}$.
\end{Remark}

\begin{Example}\textbf{1.} For $d=1$, the equation is
\[
y'-P_0' y=0.
\]

\textbf{2.} For $d=2$, the equation is
\[
y''-2P'_0\ y'+\big[ (P'_0 )^2 -P''_0\big ]\ y =0.
\]
In particular, for $P_0(z)=z^2$,
\[
y''-4z y'+\big(4z^2-2\big) y =0.
\]
\end{Example}

\subsection{Liouville classification}

Between 1830 and 1840 J.~Liouville developed a classification of transcendental functions ge\-nerated by algebraic expressions, logarithms and exponentials, and proved the non-elementary character of some natural integrals and solutions of some differential equations. Later he noticed that his classification can be extended by allowing integrations instead of using the logarithm function, which constitutes a particular case since any expression~$\log f$ is the primitive of~$f'/f$.

We recall Liouville's classification. Functions of order $0$ are algebraic functions of the variab\-le~$z$, that is those functions satisfying a polynomial equation with polynomial coefficients on~$z$. Assume by induction that order $n$ functions have been defined. Functions of order $n+1$ are those functions that are not of order~$n$ and that can be obtained by taking an exponential or a~primitive of order $n$ functions or that satisfy an algebraic equation with such coefficients.

We refer to J.F.~Ritt's book on elementary integration~\cite{Ritt1} for more information on this subject, the precursor of modern differential algebra.

Note that Liouville classification only concerns functions that are multivalued in the complex plane, i.e., except for isolated singularities and ramifications they can be continued holomorphically through all the complex plane when avoiding these isolated singularities (these are called ``fluent'' functions in Ritt's terminology~\cite{Ritt1}).

From this classification we have:

\begin{Proposition}Entire functions in the ring $\mathbb{A}_{P_0}$ are functions of order at most $2$. Moreover, if $d\geq 2$, we have that $F_0$ is of order $2$.
\end{Proposition}

For the proof of the non-elementarity of the integral giving $F_0$
see \cite[p.~48]{Ritt1}.

\section{The Ramificant determinant}\label{section2}

\subsection{Definition of the Ramificant determinant}
From now on we normalize $P_0$ to have leading coefficient $-1/d$. We denote $\omega_1, \dots ,\omega_d$ the $d$ roots of $1$, for $k=1,\dots , d$,
\[
\omega_k = {\rm e}^{\frac{2\pi}{d} {\rm i} (k-1)}.
\]
From the normalization of $P_0$, the functions $F_k$ have $d$ asymptotic
values in the directions given by the $(\omega_l)$. We denote these
values by
\[
\Omega_{kl} = \Omega_{kl}(P_0)=F_k(+\infty .\omega_l) =\int_0^{+\infty .\omega_l} t^{k-1} {\rm e}^{P_0(t)} \, {\rm d}t.
\]
These asymptotic values are transcendental periods (see~\cite{KZ2001} for the terminology), and also locations of infinite ramification points in the associated log-Riemann surfaces. They have a deep transalgebraic meaning.

\begin{Definition}The Ramificant determinant associated to $P_0$ is
\[
\Delta (P_0)=\left |
\begin{matrix}
\displaystyle \int_0^{+\infty .\omega_1} {\rm e}^{P_0(z)}\, {\rm d}z &
\displaystyle\int_0^{+\infty .\omega_1} z {\rm e}^{P_0(z)} \, {\rm d}z & \dots &\displaystyle \int_0^{+\infty .\omega_1} z^{d-1} {\rm e}^{P_0(z)} \, {\rm d}z \\
\displaystyle\int_0^{+\infty .\omega_2} {\rm e}^{P_0(z)} \, {\rm d}z &\int_0^{+\infty
.\omega_2} z {\rm e}^{P_0(z)} \, {\rm d}z & \dots &\displaystyle \int_0^{+\infty .\omega_2}
z^{d-1} {\rm e}^{P_0(z)} \, {\rm d}z \\
\vdots &\vdots &\ddots & \vdots \\
\displaystyle \int_0^{+\infty .\omega_d} {\rm e}^{P_0(z)} \, {\rm d}z & \displaystyle \int_0^{+\infty
.\omega_d} z {\rm e}^{P_0(z)}\, {\rm d}z & \dots &\displaystyle \int_0^{+\infty .\omega_d}
z^{d-1}{\rm e}^{P_0(z)} \, {\rm d}z
\end{matrix}
\right |.
\]
\end{Definition}

If we write
\[
P_0(t)=-\frac{1}{d} t^d+a_{d-1} t^{d-1} +\dots +a_1 t +a_0
\]
with $(a_0 ,a_1,\dots , a_{d-1}) \in \mathbb{C}^d$ then the Ramificant determinant is an entire function of $d$ complex variables and we write
\[
\Delta(P_0)=\Delta (a_0,a_1,\dots ,a_{d-1})
\]
and
\[
\Omega_{kl}(a_0,a_1,\dots ,a_{d-1}) =\Omega_{kl}(P_0).
\]

\subsection{Formula for the Ramificant determinant}
Even if we cannot compute in general the asymptotic values, it turns out that we can compute the Ramificant determinant. We have the following important result:

\begin{Theorem}\label{thm:main_ramificant}For each $d\geq 0$, there exists a universal polynomial of $d$ variables with rational coefficients
\[
\Pi_d (X_0,X_1, \dots , X_{d-1})\in \mathbb{Q}[X_0,\dots , X_{d-1}]
\]
with $\Pi_d (0,\dots, 0)=0$ and such that the Ramificant determinant is given by
\[
\Delta (a_0,a_1,\dots , a_{d-1}) =\frac{\left ( 2\pi d\right )^{\frac{d}{2}}}{\sqrt {2\pi}}
\ \exp\left ({\Pi_d (a_0 ,a_1 ,\dots , a_{d-1}) } \right ) .
\]
\end{Theorem}

A fundamental corollary of this theorem is that the Ramificant determinant is never $0$.

\begin{Corollary}The Ramificant determinant does not vanish
\[
\Delta (a_0,a_1,\dots , a_{d-1}) \not= 0 .
\]
\end{Corollary}

The miracle of the theorem is that among the parameter space $\mathbb{C}^d$ there is exactly one point, namely $(0,\dots , 0)$, where we can
explicitly, compute the Ramificant determinant. Then from $\Delta (0,\dots , 0)$ we derive the general formula for $\Delta (a_0,a_1,\dots , a_d)$. We first
compute the period for $P_0(t)=-\frac{1}{d}t^d$.

\begin{Lemma}Let $\omega$ be a $d$-root of $1$, $\omega^d=1$. We have
\[
\int_0^{+\infty . \omega} t^k {\rm e}^{-t^d / d} \,{\rm d}t = \omega^{k+1} d^{\frac{k+1}{d} - 1}\Gamma\left ( \frac{k+1}{d}\right ),
\]
i.e.,
\[
\Omega_{kl}(0,\dots , 0)= \omega_l^{k+1} d^{\frac{k+1}{d} - 1}\Gamma\left ( \frac{k+1}{d}\right ).
\]
\end{Lemma}

\begin{proof}By a linear change of variables we have
\[
\int_0^{+\infty . \omega} t^k {\rm e}^{-t^d / d} \,{\rm d}t =\omega^{k+1} \int_0^{+\infty} t^k {\rm e}^{-t^d / d} \,{\rm d}t.
\]
Now, the change of variables $u=s^d / d$ gives
\begin{gather*}
\omega^{k+1} \int_0^{+\infty} t^k {\rm e}^{-t^d / d} \,{\rm d}t= \omega^{k+1}
d^{\frac{k+1}{d} - 1} \int_0^{+\infty } u^{\frac{k+1}{d}-1} {\rm e}^{-u}\,{\rm d}u =
\omega^{k+1} d^{\frac{k+1}{d} - 1}\Gamma\left ( \frac{k+1}{d}\right ). \tag*{\qed}
\end{gather*}
\renewcommand{\qed}{}
\end{proof}

Now we compute $\Delta (0,\dots , 0)$.

\begin{Lemma}\label{lemma:computation_Delta(0)}We have
\[
\Delta (0,\dots , 0) =\frac{( 2\pi d)^{\frac{d}{2}}}{\sqrt {2\pi}}.
\]
\end{Lemma}

\begin{proof}Using the previous lemma we have
\begin{align*}
\Delta (0,\dots , 0) &=\left |
\begin{matrix}
\displaystyle d^{\frac{1}{d}-1} \Gamma \left (\frac{1}{d}\right ) \omega_1 &
\displaystyle d^{\frac{2}{d}-1} \Gamma \left (\frac{2}{d}\right ) \omega_1^2 &
\dots &
\displaystyle d^{\frac{d}{d}-1} \Gamma \left (\frac{d}{d}\right ) \omega_1^d \vspace{1mm}\\
\displaystyle d^{\frac{1}{d}-1} \Gamma \left (\frac{1}{d}\right ) \omega_2 &
\displaystyle d^{\frac{2}{d}-1} \Gamma \left (\frac{2}{d}\right ) \omega_2^2 &
\dots &
\displaystyle d^{\frac{d}{d}-1} \Gamma \left (\frac{d}{d}\right ) \omega_2^d \\
\vdots &\vdots &\ddots & \vdots \\
\displaystyle d^{\frac{1}{d}-1} \Gamma \left (\frac{1}{d}\right ) \omega_d &
\displaystyle d^{\frac{2}{d}-1} \Gamma \left (\frac{2}{d}\right ) \omega_d^2 &
\dots &
\displaystyle d^{\frac{d}{d}-1} \Gamma \left (\frac{d}{d}\right ) \omega_d^d \\
\end{matrix}
\right | \\
&=d^{\frac{1}{d} (1+2+\dots +d)-d}
\Gamma \left (\frac{1}{d}\right )
\Gamma \left (\frac{2}{d}\right )\dots
\Gamma \left (\frac{d}{d}\right ) \left |
\begin{matrix}
\omega_1 & \omega_1^2 & \dots & \omega_1^d \\
\omega_2 & \omega_2^2 & \dots & \omega_2^d \\
\vdots &\vdots &\ddots & \vdots \\
\omega_d & \omega_d^2 & \dots & \omega_d^d \\
\end{matrix}
\right | \\
&=d^{\frac{1-d}{2}} (2\pi)^{\frac{d-1}{2}} d^{\frac{1}{2} -d\frac{1}{d}}
\Gamma (1) \left |
\begin{matrix}
\omega_1& \omega_1^2 & \dots & \omega_1^d \\
\omega_2 &\omega_2^2 & \dots & \omega_2^d \\
\vdots &\vdots &\ddots & \vdots \\
\omega_d & \omega_d^2 & \dots & \omega_d^d \\
\end{matrix}
\right | \\
&=\frac{1}{\sqrt {2\pi}} \left ( \frac{2\pi}{d} \right )^{\frac{d}{2}}
\left |
\begin{matrix}
\omega_1 & \omega_1^2 & \dots & \omega_1^d \\
\omega_2 & \omega_2^2 & \dots & \omega_2^d \\
\vdots &\vdots &\ddots & \vdots \\
\omega_d & \omega_d^2 & \dots & \omega_d^d \\
\end{matrix}
\right | ,
\end{align*}
where we have used Gauss multiplication formula

\[
\Gamma (z ).\Gamma \left ( z+\frac{1}{d}\right )\dots
\Gamma \left (z+\frac{d-1}{d}\right )=(2\pi)^{\frac{d-1}{2}} d^{\frac{1}{2}-dz}
\Gamma(dz).
\]
We have that $\omega_j^d=1$ and the last determinant is equal to $(-1)^{d-1} V_d$ where $V_d$ is the Vandermonde determinant
\[
V_d=\left |\begin{matrix}
1 &\omega_1 &\omega_1^2 &\dots &\omega_1^{d-1} \\
1 &\omega_2 &\omega_2^2 &\dots &\omega_2^{d-1} \\
\vdots &\vdots&\vdots &\ddots &\vdots \\
1 &\omega_d &\omega_d^2 &\dots &\omega_d^{d-1} \\
\end{matrix}
\right |=\prod_{i\not= j} (\omega_i-\omega_j).
\]
Finally, the next lemma applied to the polynomial $Q(X)=X^d-1$, shows that
\begin{gather*}
V_d=\prod_i \big(d\omega_i^{d-1}\big)=d^d \bigg( \prod_i \omega_i \bigg)^{d-1}=(-1)^{d-1} d^d. \tag*{\qed}
\end{gather*}
\renewcommand{\qed}{}
\end{proof}

\begin{Lemma}If $\xi_1 ,\dots , \xi_d$ are the $d$ roots
of a monic polynomial $Q(X)$, then we can compute the
Vandermonde determinant $V(\xi_1 ,\dots , \xi_d)$ of the $(\xi_1 ,\dots , \xi_d)$ as
\[
V(\xi_1 ,\dots , \xi_d)= \left |
\begin{matrix}
1 &\xi_1 &\xi_1^2 &\dots &\xi_1^{d-1} \\
1 &\xi_2 &\xi_2^2 &\dots &\xi_2^{d-1} \\
\vdots &\vdots&\vdots &\ddots &\vdots \\
1 &\xi_d &\xi_d^2 &\dots &\xi_d^{d-1} \\
\end{matrix}
\right |
=\prod_{i\not= j} (\xi_i-\xi_j)=\prod_{i=1}^d Q'(\xi_i ).
\]
\end{Lemma}

\begin{proof}We have $Q'(\xi_i)=\prod_{j\not=i} (\xi_i-\xi_j)$ and the result follows.
\end{proof}

Now we can prove Theorem \ref{thm:main_ramificant}.

\begin{proof}[Proof of Theorem \ref{thm:main_ramificant}]
Consider the entire function of several complex variables $\Delta (a_0,a_1,\dots ,\allowbreak a_{d-1})$.
Observe that by Theorem \ref{thm:computation_integrals} we have that
each integral
\[
\int_0^{+\infty . \omega_i} z^n {\rm e}^{P_0(z)} \,{\rm d}z
\]
is a linear combination with coefficients polynomial integer
coefficients on the $(a_j)$ of the integrals for $j=0,1,\dots,d-1$,
\[
\int_0^{+\infty . \omega_i} z^j {\rm e}^{P_0(z)}\,{\rm d} z.
\]
Therefore, differentiating column by column, we observe that for each $j=0,1,\dots , d-1$, we have
\[
\partial_{a_j} \Delta = c_j \Delta,
\]
where $c_j$ is a polynomial on the $(a_j)$ with integer
coefficients. We conclude that the logarithmic derivative of
$\Delta$ with respect to each variable is a universal polynomial
with integer coefficients on the variables $(a_j)$. This gives the
existence of the universal polynomial $\Pi_d$ such that
\[
\Delta (a_0,a_1,\dots ,a_{d-1})=c.{\rm e}^{\Pi_d(a_0,a_1,\dots ,a_{d-1})},
\]
with $\Pi_d (0,\dots , 0)=0$ and $c=\Delta (0,\dots , 0) \in \mathbb{C}$. The result follows from Lemma~\ref{lemma:computation_Delta(0)}.
\end{proof}

\subsection[The universal polynomials $\Pi_d$]{The universal polynomials $\boldsymbol{\Pi_d}$}

It is interesting to compute and study the combinatorial properties of the family of universal polynomials $(\Pi_d)$. We can compute a few first polynomials.

\begin{Theorem}We have
\begin{gather*}
\Pi_1 (X_0) =X_0,\qquad
\Pi_2 (X_0,X_1) = 2X_0+\frac{1}{2} X_1^2,\\
\Pi_3 (X_0,X_1,X_2) = 3X_0 + 2 X_1 X_2 +\frac{4}{3} X_2^3,
\end{gather*}
and for $d=4$
\[
\Pi_4(X_0,X_1,X_2,X_3)=4X_0 +3X_3 X_1+ 2 X_2^2+9X_3^2X_2+\cdots,
\]
where the remaining term is a polynomial in $X_3$, and, in general,
for $d\geq 5$,
\[
\Pi_d (X_0,X_1,\dots , X_{d-1})=dX_0 +(d-1)X_{d-1} X_1+ \big(2 (d-2) X_{d-2} + (d-1)^2 X_{d-1}^2\big) X_2 +\cdots,
\]
where the remaining terms are independent of $X_0$, $X_1$ and $X_2$.

More generally, $\Pi_d$ is of degree $1$ in $X_k$ for $k< d/2$.
\end{Theorem}

\begin{proof}For $d\geq 1$ the dependence of the Ramificant determinant $\Delta$ on $a_0$ is
straightforward by direct factorization of ${\rm e}^{a_0}$ in the integrals, which gives
\[
\Pi_d (X_0,\dots ,X_{d-1}) =dX_0+\cdots
\]
with remaining terms are independent of $X_0$. Also this can be seen by
differentiation column by column of $\Delta$,
\[
\partial_{a_0} \Delta =d\Delta,
\]
which also gives the result.
For the dependence on $a_1$ we use this last approach. For $d\geq 2$,
we have
\[
\partial_{a_1} \Delta = (d-1) a_{d-1} \Delta.
\]
This is because the differentiation of the first $d-1$ columns yields $0$. Also for the last column we have
\[
z^d =-zP'_0(z)+(d-1)a_{d-1}z^{d-1}+(d-2)a_{d-2}z^{d-2}+\dots +a_1 z.
\]
And the integrals corresponding to the term $-zP'_0(z)$ contribute
$0$ because
\[
\int -z P'_0(z){\rm e}^{P_0(z)} \,{\rm d}z =\big[{-}z{\rm e}^{P_0}\big]+\int {\rm e}^{P_0(z)}\,{\rm d} z.
\]
And by linearity of the integrals in the last column the lower order terms $(d-2)a_{d-2}z^{d-2}+\dots +a_1 z$ contribute $0$. Thus the only contribution comes from the term $(d-1)a_{d-1}z^{d-1}$ which gives $(d-1)a_{d-1} \Delta $. Now this last equation gives for $d=2$,
\[
\partial_{a_1} \Delta = a_{1} \Delta,
\]
and we have $\Pi_2 (X_0,X_1)=2X_0 +\frac{1}{2}X_1^2 $.

For $d\geq 3$ we get
\[
\Pi_d (X_0,X_1,\dots , X_{d-1})=dX_0+(d-1)X_{d-1}X_1+\cdots,
\]
where the remaining terms are independent of $X_0$ and $X_1$. Now we assume $d\geq 3$ and we~determine the dependence on~$a_2$. We proceed as before and differentiate column by co\-lumn~$\partial_{a_2} \Delta$. Only the last two columns give a~contribution. The last but one column contributes by $(d-2)a_{d-2} \Delta$
because
\[
z^d =-zP'_0(z)+(d-1)a_{d-1}z^{d-1}+(d-2)a_{d-2}z^{d-2}+\dots +a_1 z,
\]
and the last column contributes by $\big[(d-2)a_{d-2}\Delta+(d-1)^2a_{d-1}^2\big] \Delta$ because
\[
z^{d+1}=-z^2 P_0'(z)+(d-1)a_{d-1}z^{d}+(d-2)a_{d-2}z^{d-1}+\dots
+a_1 z^2,
\]
and modulo $P'_0$ we have
\[
z^{d+1}=\big[(d-2)a_{d-2}\Delta+(d-1)^2a_{d-1}^2\big] z^{d-1}+\cdots +[P_0'],
\]
where the dots denote lower order terms. Thus we have
\[
\partial_{a_2} \Delta=\big(2(d-2)a_{d-2}+(d-1)^2a_{d-1}^2\big)\Delta.
\]
When $d=3$ this gives
\[
\partial_{a_2} \Delta=\big(2a_1+4 a_2^2\big)\Delta,
\]
therefore
\[
\Pi_3 (X_0,X_1,X_2)=3X_0 +2X_2X_1 +\frac{4}{3} X_2^3.
\]
When $d=4$ we get
\[
\partial_{a_2} \Delta=\big(4 a_2+9 a_3^2\big)\Delta.
\]
So
\[
\Pi_4(X_0,X_1,X_2,X_3)=4X_0 +3X_3 X_1+ 2 X_2^2+9X_3^2X_2+\cdots,
\]
where the remaining term is a polynomial in $X_3$.
When $d\geq 5$ we get
\begin{gather*}
\Pi_d (X_0,X_1,\dots , X_{d-1})=dX_0\! +(d-1)X_{d-1} X_1 + \big( 2 (d-2) X_{d-2}\!+ (d-1)^2 X_{d-1}^2\big) X_2\! +\cdots,
\end{gather*}
where the remaining terms are independent of $X_0$, $X_1$ and $X_2$.

A close inspection of the procedure (for a complete analysis see
what follows next) shows that if $k<d/2$ then
\[
\partial_{a_k} \Delta =c \Delta,
\]
where $c$ is a polynomial on $a_{d-1},a_{d-2},\dots , a_{d-k}$ thus
the last result follows.
\end{proof}

The next results provide an algorithm to compute the universal polynomial $\Pi_d$.

\begin{Theorem}Let $d\geq 2$. For $n\geq 0$ we define $(A_{n,k})_{0\leq k\leq d-1}$ to be the coefficients of the remainder when dividing $z^n$ by $zP_0'$:
\[
z^n =A_{n,d-1} z^{d-1}+A_{n,d-2}z^{d-2}+\dots +A_{n,1} z+A_{n,0} \
[zP_0'].
\]
For $n\leq d-1$ and $k\not=n$, we have $A_{n,k}=0$, and $A_{n,n}=1$.

For $n=d$,
\[
A_{d,k}= ka_k .
\]

And for $n\geq d+1$, we can compute the sequence $(A_{n,k})$ by induction using
\[
A_{n+1,k}=(d-1)a_{d-1} A_{n,k}+(d-2) a_{d-2} A_{n-1,k}+\dots + a_1 A_{n-d+2,k}.
\]
\end{Theorem}

\begin{proof}For the induction relation, we use
\[
z^{n+1} =-z^{n-d+2} P'_0 +(d-1)a_{d-1} z^n+ (d-2) a_{d-2} z^{n-1}+\dots +a_1 z^{n-d+2}.
\]
The rest is clear.
\end{proof}

\begin{Corollary} For $d\geq 2$, $0\leq k\leq d-1$, and $n\geq d$, $A_{n,k}$ is a polynomial with integer coefficients on $a_0,a_1,\dots , a_{d-1}$ of total degree $n-d+1$.
\end{Corollary}

\begin{proof}This is straightforward from the induction relations.
\end{proof}

Now we can compute the polynomial $\Pi_d$ using the polynomials $(A_{n,k})$.

\begin{Corollary}For $d\geq 2$, the polynomial $\Pi_d$ is uniquely determined by the equations, for $0\leq k\leq d-1$,
\[
\partial_{a_k} \Pi_d (a_0,\dots , a_{d-1}) =A_{d-1+k,d-1}+A_{d-2+k,d-2}+\dots+ A_{d,d-k}.
\]
\end{Corollary}

\begin{proof}Differentiating, column by column, we get (this is clear from the above computations)
\[
\partial_{a_k} \Delta =\big( A_{d-1+k,d-1}+A_{d-2+k,d-2}+\dots+A_{d,d-k}\big ) \Delta,
\]
and the result follows.
\end{proof}

\section{Applications of the Ramificant determinant}\label{section3}

\subsection{Integrability and Abel-like theorem}

The non-vanishing of the Ramificant determinant immediately gives the following result:

\begin{Theorem}\label{thm:vanishing_periods}In the $\mathbb{C}$-vector space $\langle F_0, \dots , F_{d-1}\rangle_\mathbb{C} $ the only function with all asymptotic values vanishing is the $0$ function. In the $\mathbb{C}$-vector space $\mathbb{U}_{P_0}=\langle 1,F_0, \dots , F_{d-1}\rangle_\mathbb{C}$ the subspace of functions with vanishing
asymptotic values is the complex line generated by ${\rm e}^{P_0}$.
\end{Theorem}

A primitive $\int Q {\rm e}^{P_0}\,{\rm d}t$ is integrable in finite terms in the sense of Abel and Liouville if we can compute this primitive and it is an element of the ring $\mathbb{C}\big[z, {\rm e}^{P_0}\big]$. Therefore, in this context of elementary integration, we say that an holomorphic $1$-form $\omega$ is exact if
there is a function $f\in \mathbb{C}\big[z, {\rm e}^{P_0}\big]$ such that ${\rm d}f=\omega$.\footnote{We thank the referee for pointing out this precision to avoid
confusion with the usual notion of exact form in differential geometry.} We give a simple criterion for integrability in finite terms.

\begin{Theorem}[integrability criterion]\label{thm:integrability} A necessary and sufficient condition for a primitive
\[
F(z)=\int_0^z Q(t) {\rm e}^{P_0(t)} \,{\rm d}t
\]
to be computable in finite terms is that the $d$ asymptotic values for $l=1,\dots ,d$,
\[
F(+\infty.\omega_l)=\Omega_l(F)=\int_0^{+\infty.\omega_l} Q(t) {\rm e}^{P_0(t)} \,{\rm d}t
\]
are all the same constant $\Omega(F)$.

In that case, the differential $Q(t){\rm e}^{P_0(t)} \,{\rm d}t$ is exact,
\[
Q(t){\rm e}^{P_0(t)}\,{\rm d}t = {\rm d} \big( A(t) {\rm e}^{P_0(t)} \big)
\]
for some $A\in \mathbb{C}[t]$ such that $AP_0'+A' =Q$.
\end{Theorem}

\begin{proof}Note that we know from Theorem \ref{thm:computation_integrals} that such a function $F$ is of the form
\begin{gather*}
F(z) = A_0(z){\rm e}^{P_0(z)} + b_{-1}+ b_0 F_0(z)+\dots +b_{d-1} F_{d-1}(z),
\end{gather*}
where $A_0\in \mathbb{C}[z]$, $A_0(0)=0$, and $b_{-1}, b_0,\dots , b_{d-1} \in \mathbb{C}$. Making $z=0$, we have
$b_{-1} = 0$ and
\begin{gather}\label{eq:function_F}
F(z) = A_0(z){\rm e}^{P_0(z)} + b_0 F_0(z)+\dots +b_{d-1} F_{d-1}(z).
\end{gather}
So we have for $l=1,\dots , d$,
\begin{gather}\label{eq:linear_system}
\sum_{k=0}^{d-1} b_k \Omega_{kl} =\Omega(F).
\end{gather}
We can look at these equations as a linear system on $(b_0,\dots ,b_{d-1})$.
The non-vanishing of the Ramificant determinant shows that there is exactly one solution. But if
we choose $(b_0, b_1, \dots ,\allowbreak b_{d-2}, b_{d-1})=(a_1, 2a_2,\dots, (d-1)a_{d-1}, -1)$, then we have for
$l=1,\dots , d$,
\[
\sum_{k=0}^{d-1} b_k \Omega_{kl} =\big[{\rm e}^{P_0(t)}\big]_0^{+\infty.\omega_l}= -{\rm e}^{P_0(0)}.
\]
Therefore, the only solution to the system \eqref{eq:linear_system} is
\[
(b_0, b_1, \dots , b_{d-2}, b_{d-1})=(a_1, 2a_2,\dots, (d-1)a_{d-1}, -1).\big({-}\Omega(F) {\rm e}^{-P_0(0)} \big)
\]
and plugging this value in equation \eqref{eq:function_F}, we get
\begin{gather*}
F(z) = A_0(z){\rm e}^{P_0(z)} + \big[{\rm e}^{P_0(t)}\big]_0^{z}\big({-}\Omega(F) {\rm e}^{-P_0(0)} \big)
 =\big(A_0(z)-\Omega(F) {\rm e}^{-P_0(0)}\big){\rm e}^{P_0(z)} +\Omega(F),
\end{gather*}
thus $F$ is computable in finite terms. The exactness of the differential follows by differentiation of this equation with $A(t)=A_0(t)-\Omega(F) {\rm e}^{-P_0(0)}$.
\end{proof}

This result can be reformulated as an Abel's theorem in this setting. We consider paths $(\gamma_l)_{1\leq l\leq d}$ going to $\infty$ in $\mathbb{C}$ starting in the direction given by $\omega_l$ and ending in the direction given by $\omega_{l+1}$ (the index $l$ is taken modulo $d$). Then we consider the transcendental periods
\[
\int_{\gamma_l} Q(t) {\rm e}^{P_0(t)} \,{\rm d}t =F(+\infty. \omega_{l+1})-F(+\infty. \omega_{l}).
\]
The condition of Theorem \ref{thm:integrability} that all asymptotic values are equal is equivalent to have all periods vanishing
\[
\int_{\gamma_l} Q(t) {\rm e}^{P_0(t)} \,{\rm d}t = 0
\]
and then the conclusion of Theorem \ref{thm:integrability} is that the differential form $\omega = Q(t) {\rm e}^{P_0(t)} \,{\rm d}t$ is exact.
Note that the integral over the path $\gamma_l$ only depends on the homotopy class of $\gamma_l$ relative to the asymptotic directions.
The converse is clear: If the holomorphic differential form $\omega$ is exact then all periods are zero, for $1\leq l\leq d$ we have
\[
\int_{\gamma_l} \omega \,{\rm d} t = 0.
\]
The $\mathbb{C}$-vector space $\mathcal{H}^1$ of holomorphic differential forms of the type $\omega=Q(t) {\rm e}^{P_0(t)}\,{\rm d}t$ modulo exact differentials (de Rham cohomology space-type) has a base (Abelian differentials)
\begin{gather*}
\omega_0 = {\rm e}^{P_0(t)} \,{\rm d}t, \qquad \omega_1 = t{\rm e}^{P_0(t)} \,{\rm d}t, \qquad \dots, \qquad \omega_{d-1}= t^{d-1} {\rm e}^{P_0(t)} \,{\rm d}t.
\end{gather*}
We consider the $\mathbb{C}$-vector space $\mathcal{H}_1$ of formal $\mathbb{C}$-linear combination of paths $(\gamma_l)_{1\leq l\leq d}$ ($\mathbb{C}$-homology space).
What we proved is the following Abel-like theorem:

\begin{Theorem}[Abel-like theorem]The pairing $\mathcal{H}^1 \times \mathcal{H}_1 \to \mathbb{C}$ given by
\[
(\omega , \gamma) \mapsto \int_\gamma \omega\,{\rm d}t
\]
is non-degenerate.
\end{Theorem}

A generalization of this result to non-simply connected finite type log-Riemann surfaces is proved in \cite{biderham}.

\subsection{The period mapping is \'etale}

\begin{Definition}The period mapping $\Upsilon\colon \mathbb{C}^d \to \mathbb{C}^d$ is
\[
\Upsilon (a_0,a_1 ,\dots , a_{d-1})=(F_0(+\infty .\omega_1),F_0(+\infty .\omega_2),\dots ,
F_0(+\infty .\omega_d)).
\]
\end{Definition}

\begin{Theorem}The period mapping $\Upsilon$ is a local diffeomorphism everywhere.
\end{Theorem}

\begin{Remark}The period mapping is not a global diffeomorphism as is easily seen construc\-ting two distinct log-Riemann surfaces with $d$ ramification points with the same images by the projection mapping $\pi$.
\end{Remark}

\begin{proof} The computation of the determinant of the differential of the period mapping at a point gives the value of the Ramificant determinant at this point,
\[
\det D_{a_0,\dots ,a_{d-1}} \Upsilon =\Delta (a_0,\dots , a_{d-1}).
\]
Then we use the local inversion theorem using the non-vanishing of the determinant.
\end{proof}

\subsection{Separation of asymptotic directions}

Using the functions $F_0, \dots , F_{d-1}$ we can distinguish the different asymptotic directions.

\begin{Theorem}\label{thm:separation_directions}Let $\omega_k$ and $\omega_l$ be roots of $1$ such that for all $j=0,1,\dots ,d-1$, we have
\[
F_j(+\infty.\omega_k) =F_j(+\infty.\omega_l)
\]
then
\[
\omega_k=\omega_l.
\]
\end{Theorem}

\begin{proof}Otherwise the Ramificant determinant will have two identical rows and will vanish.
\end{proof}

\subsection{Transalgebraic symmetric formulas}

The natural transalgebraic philosophy is to think of the transcendental periods $(F_k(+\infty .\omega_l))$ as transalgebraic numbers when $P_0(z)\in \mathbb{Q} [z]$. Then it is
natural to ask what is the relation between these periods and the coefficients of $P_0$ that define them, similar to the fundamental symmetric formulas for the roots of an algebraic equation. We have the following:

\begin{Theorem}For $j=1,\dots , d-1$ $($note that $j=0$ is excluded$)$, we have that ${\rm e}^{-a_0} a_j$ is a~universal rational function on
\[
(F_k(+\infty . \omega_l))_{k=0,\dots , d\atop l=1,\dots ,d}.
\]
More precisely, ${\rm e}^{-a_0} a_j\Delta$ $($where $\Delta$ is the Ramificant determinant$)$ is a universal polynomial function of degree $d-1$ on $(F_k(+\infty . \omega_l))_{k=0,\dots , d\atop l=1,\dots ,d}$.
\end{Theorem}

\begin{proof}Observe that for $l=1,\dots ,d,$ we have
\begin{gather*}
 - F_{d-1}(+\infty .\omega_l) + (d-1) a_{d-1}
F_{d-2}(+\infty . \omega_l)+\dots +a_1 F_0(+\infty .\omega_l) \\
\qquad {} =\int_0^{+\infty . \omega_l} P'_0(z) {\rm e}^{P_0(z)} \,{\rm d}z =\big[{\rm e}^{P_0(z)}\big]_0^{+\infty .\omega_l} =-{\rm e}^{a_0}.
\end{gather*}
Therefore if we consider the matrix
\[
M =\begin{pmatrix}
F_0(+\infty . \omega_1) & F_1(+\infty . \omega_1) & \dots &F_{d-1}(+\infty . \omega_1) \\
F_0(+\infty . \omega_2) &F_1(+\infty . \omega_2) & \dots &F_{d-1}(+\infty . \omega_2) \\
\vdots &\vdots &\ddots & \vdots \\
F_0(+\infty . \omega_d) & F_1(+\infty . \omega_d) & \dots &F_{d-1}(+\infty . \omega_d)
\end{pmatrix}
\]
we have
\[
M.\begin{pmatrix}
a_1 \\
2 a_2 \\
\vdots \\
(d-1)a_{d-1} \\
-1 \\
\end{pmatrix}
= -{\rm e}^{a_0}
\begin{pmatrix}
1 \\
1 \\
\vdots \\
1 \\
1 \\
\end{pmatrix}.
\]
Thus
\[
\begin{pmatrix}
a_1 \\
2 a_2 \\
\vdots \\
(d-1)a_{d-1} \\
-1 \\
\end{pmatrix}
= -{\rm e}^{a_0} M^{-1}
\begin{pmatrix}
1 \\
1 \\
\vdots \\
1 \\
1 \\
\end{pmatrix}
\]
and by Cramer's formulas the coefficients of $M^{-1}$ are polynomials on the entries of $M$ divided by the Ramificant $\Delta =\det M$.
\end{proof}

\subsection{Torelli-like theorem}

As we have observed, only the location of the ramification points, i.e., the values $(F_0(+\infty . \omega_l))$ are not enough to characterize the polynomial $P_0$ (or the associated log-Riemann surface). This changes if we consider all values $(F_k(+\infty . \omega_l))$ as the next corollary shows. Thus we obtain that the periods determine the log-Riemann surface, which is a Torelli-like theorem.

\begin{Corollary}[Torelli-like theorem] Let $P_0$ and $Q_0$ be two normalized polynomials,
\begin{gather*}
P_0(z) =-\frac{1}{d} z^d +a_{d-1} z^{d-1} +\dots +a_1 z +a_0, \\
Q_0(z) =-\frac{1}{d} z^d +b_{d-1} z^{d-1} +\dots +b_1 z +b_0.
\end{gather*}
Consider the associated functions
\begin{gather*}
F_k(z) =\int_0^z t^k {\rm e}^{P_0(t)} \,{\rm d}t,\qquad
G_k(z) =\int_0^z t^k {\rm e}^{Q_0(t)} \,{\rm d}t.
\end{gather*}
If for $k=0,\dots , d-1$ and $l=1,\dots , d$ we have
\[
F_k(+\infty . \omega_l)=G_k (+\infty . \omega_l),
\]
then for $k=1,\dots, d-1$, we have ${\rm e}^{a_0}a_k={\rm e}^{b_0} b_k$, i.e.,
\[
{\rm e}^{P_0(0)} (P_0(z)-P_0(0))={\rm e}^{Q_0(0)} (Q_0(z)-Q_0(0)).
\]
So the polynomials are determined up to their constant term. In particular, if the polynomials have the same constant term, then
\[
P_0=Q_0.
\]
\end{Corollary}

\section{Introduction to transalgebraic Dedekind--Weber theory}\label{section4}

\subsection{Transalgebraic curves of genus 0}

We refer to \cite{bipm1,bipm2,bipmarxiv} for background on log-Riemann surfaces.

\begin{Definition}A transalgebraic curve $\mathcal{S}$ of genus $0$ is a simply connected log-Riemann surface with a finite set of ramification points.
\end{Definition}

Then the underlying Riemann surface is parabolic and biholomorphic to $\mathbb{C}$ (see \cite{bipm2} or \cite{bipmarxiv} for a proof). We prove in \cite{bipm2} the following basic uniformization theorem

\begin{Theorem}Let $\mathcal{S}$ be a transalgebraic curve of genus $0$, and $z_0\in \mathcal{S}$ a base point with $\pi(z_0)=0$. Let
$\tilde F\colon \mathbb{C} \to \mathcal{S}$ be the unique uniformization such that $\tilde F(0)=z_0$ and $F'(0)=1$. Then we have that
\[
F(z)= \pi\circ \tilde F (z) =\int_0^z Q(t) {\rm e}^{P_0(t)} \,{\rm d}t
\]
for some polynomials $Q, P_0 \in \mathbb{C}[t]$. The number of finite $($resp.\ infinite$)$ ramification points is $\deg Q$ $($resp.~$\deg P)$.
\end{Theorem}

From now on we consider a transalgebraic curve $\mathcal{S}$ of genus $0$ without finite ramification points, corresponding to polynomials $Q=1$ and
$P=P_0$ so that its uniformization is the lift of~$F_0$. The degree of $P_0$ is $d$ and $\mathcal{S}$ has exactly $d$ distinct infinite ramification
points that project by $\pi$ to finite values on $\mathbb{C}$ that are equal to the asymptotic values of~$F_0$.

\subsection{The structural ring}

We define a ring of functions that play the same role for $\mathcal{S}$ than polynomials for the complex plane $\mathbb{C}$.

\looseness=-1 Let $P_0\in \mathbb{C}[z]$ be the polynomial such that $d=\deg P_0$ and the uniformization of $\mathcal{S}$ is the lift of
\[
F_0(z)=\int_0^z {\rm e}^{P_0(t)} \,{\rm d}t,
\]
i.e., the uniformization $\tilde F_0\colon (\mathbb{C} , 0) \to (\mathcal{S} , z_0)$ is such that $F_0=\pi\circ \tilde F_0$. We define as in Section~\ref{section1} the transcendental functions $F_1, \dots , F_{d-1}$, and the ring $\mathbb{A}_{P_0}$ and its field of fractions~$\mathbb{K}_{P_0}$. We consider the natural sub-ring of $\mathbb{A}_{P_0}$ of holomorphic functions in $\mathcal{S}$ and having finite asymptotic values, i.e., finite functions in ${\mathcal{S}}^*$.

\begin{Definition}We consider the sub-ring $\hat{\mathbb{A}}_{P_0} \subset \mathbb{A}_{P_0}$
\[
\hat{\mathbb{A}}_{P_0} = z\mathbb{C}[z, F_0, \dots ,F_{d-1}] {\rm e}^{P_0(z)} \oplus \mathbb{C}[F_0, \dots , F_{d-1}],
\]
and its associated field of fractions $\hat{\mathbb{K}}_{P_0} \subset \mathbb{K}_{P_0}$.

The sub-ring $\hat{\mathbb{A}}_{P_0} \subset \mathbb{A}_{P_0}$ is the subspace of $\mathbb{A}_{P_0}$ of holomorphic functions with finite asymptotic values.
\end{Definition}

To justify this definition, observe that if
\begin{align*}
G_1 = A_1 {\rm e}^{P_0} +B_1, \qquad G_2 = A_2 {\rm e}^{P_0} +B_2
\end{align*}
with $A_1, A_2 \in z\mathbb{C}[z,F_0,\dots, F_{d-1}]$ and $B_1, B_2 \in \mathbb{C}[F_0,\dots, F_{d-1}]$, then we have
\[
F_1.F_2 = A {\rm e}^{P_0} +B
\]
with $A=A_1A_2 {\rm e}^{P_0} +A_1 B_2+A_2 B_1 \in z\mathbb{C}[z,F_0,\dots, F_{d-1}]$ (using Proposition~\ref{prop:exp_linear_combo}), and $B=B_1.B_2 \in \mathbb{C}[F_0,\dots, F_{d-1}]$, so we have a well defined sub-ring. We are discarding from ${\mathbb{A}}_{P_0}$ the non-constant polynomials that have infinite asymptotic values.

All functions in $\hat{\mathbb{A}}_{P_0}$ have finite asymptotic values since all $F_k$ do have finite asymptotic values, and any polynomial in $\mathbb{C}[z]$
appears multiplied by ${\rm e}^{P_0}$. We consider now $k_0\colon \mathcal{S} \to \mathbb{C}$, the inverse of the uniformization $\tilde F_0$, $k_0= \tilde F_0^{-1}$.

\begin{Definition}[structural ring] The structural ring $\hat{\mathcal{A}}_{\mathcal{S}}$ of the log-Riemann surface $\mathcal{S}$ is the ring of holomorphic functions $f$ on $\mathcal{S}$ of the form
\[
f=F\circ k_0,
\]
where $F\in \hat{\mathbb{A}}_{P_0}$. In particular, for $k=0,\dots, d-1$, we define the holomorphic functions $f_k\colon \mathcal{S} \to \mathbb{C}$ by
\[
f_k= F_k \circ k_0.
\]
Observe that $f_0=\pi$ is the projection mapping of $\mathcal{S}$. The structural ring $\hat{\mathcal{A}}_{\mathcal{S}}$ is an integral domain.

We define the structural field $\hat{\mathcal{K}}_{\mathcal{S}}$ to be the field of fractions of $\hat{\mathcal{A}}_{\mathcal{S}}$. Therefore we have
\begin{gather*}
\hat{\mathcal{A}}_{\mathcal{S}} \approx \hat{\mathbb{A}}_{P_0},\qquad \hat{\mathcal{K}}_{\mathcal{S}} \approx \hat{\mathbb{K}}_{P_0}.
\end{gather*}

We define in the same way $\mathcal{A}_{\mathcal{S}}$ and its field of fractions $\mathcal{K}_{\mathcal{S}}$.
\end{Definition}

\begin{Definition}The coordinate ring $\mathbb{C} [\pi]$, resp.\ field $\mathbb{C}(\pi )$, is the sub-ring of the structural ring~$\hat{\mathcal{A}}_{\mathcal{S}}$, resp.\ subfield of the structural field $\hat{\mathcal{K}}_{\mathcal{S}}$, generated by the coordinate function $\pi$.
\end{Definition}

Observe that we have
\begin{gather*}
\mathbb{C} [\pi] \approx \mathbb{C}[F_0] \subset \hat{\mathbb{A}}_{P_0}, \qquad \mathbb{C} (\pi ) \approx \mathbb{C} (F_0) \subset \hat{\mathbb{K}}_{P_0},
\end{gather*}
because elements $f$ of the coordinate ring are of the form
\[
f=F\circ k_0,
\]
with $F\in \mathbb{C} [F_0]$.

\subsection{Transcendence degree and number of infinite ramification points}

The number of infinite ramification points in the log-Riemann surface $\mathcal{S}$ can be read algebraically as the transcendence
degree of $\mathcal{K}_{\mathcal{S}}$ or $\hat{\mathcal{K}}_{\mathcal{S}}$ over $\mathbb{C} (\pi)$.

\begin{Theorem}The transcendence degree of $\hat {\mathcal{K}}_{\mathcal{S}}$ over $\mathbb{C} (\pi)$ is
\[
\big[\hat{\mathcal{K}}_{\mathcal{S}}:\mathbb{C} (\pi)\big]_{{\rm tr}} =d.
\]
\end{Theorem}

\begin{proof}We have that $[\mathbb{K}_{P_0}\colon \mathbb{C}[F_0]]_{{\rm tr}}=d$ because $1, z,F_0, \dots , F_{d-1}$ are algebraically independent.
\end{proof}

\subsection{Stolz limits and refined analytic estimates}

By Stolz limit at an infinite ramification point $w^*$ of $\mathcal {S}^*$ we understand a limit when we converge to $w^*$ remaining in a sector with vertex at $w^*$.

\begin{Proposition}\label{prop:Stolz_limits} Any function $f\in \hat{\mathcal{A}}_{\mathcal{S}}$ is Stolz continuous in $\mathcal{S}^*$, i.e., it has Stolz limits at the
infinite ramification points.
\end{Proposition}

It is enough to prove this result for $f$ in the vector space $\mathcal{V}_{\mathcal{S}} \subset \hat{\mathcal{A}}_{\mathcal{S}}$
\[
\mathcal{V}_{\mathcal{S}} =k_0 \mathbb{C}[k_0 ] \big({\rm e}^{P_0}\circ k_0 \big) \oplus \mathbb{C} . 1 \oplus
\mathbb{C} . f_0 \oplus \dots \oplus \mathbb{C} . f_{d-1},
\]
i.e., $f=F\circ k_0$ with $F\in \mathbb{V}_{P_0}$.

This Stolz continuity is weaker than continuity for the topology defined by the natural flat metric on $\mathcal{S}$ that gives the completion~$\mathcal{S}^*$. We can show that the only continuous functions in~$\mathcal{V}_{\mathcal{S}}$ for the completion topology are the ones in the coordinate sub-ring $\mathbb{C}[\pi]$. We have:

\begin{Proposition}\label{prop:Stolz_limits_2}
Any function $f \in \mathcal{V}_{\mathcal{S}}$ not belonging to the subspace $\mathbb{C} .1\oplus \mathbb{C} . f_0$
has a Stolz continuous extension to $\mathcal{S}^*$ but not a continuous extension. In particular, the
functions $f_1, \dots , f_{d-1}$ do extend Stolz continuously to $\mathcal{S}^*$ but not continuously. The function
$f_0$ also extends continuously to $\mathcal{S}^*$ for the metric topology.
\end{Proposition}

This result and a stronger version of Proposition~\ref{prop:Stolz_limits} is proved in \cite[Section~III.2]{bipmarxiv} and results
from refined analytic estimates for the functions $f\in \hat{\mathcal{A}}_{\mathcal{S}}$, but we can also prove it directly
by the same argument used in conformal representation theory to prove that the existence of a radial limit implies Stolz convergence.

\subsection{Liouville theorem} \label{subsection:Liouville}

We have growth conditions that characterize the functions in the vector space $\mathcal{V}_{\mathcal{S}}$. For the precise statement and the proof of the following theorem (that we will not use in this article) we refer to \cite[Section III.3]{bipmarxiv}.

\begin{Theorem}[general Liouville theorem]\label{thm:general_liouville} Let $f\colon \mathcal{S}\to \mathbb{C}$ be a holomorphic function which has a finite Stolz continuous extension to $\mathcal{S}^*$. Let $\infty$ the end at infinite of the Alexandrov compactification of $\mathcal{S}$. If $f$ satisfies a precise set of growth conditions on $f(w)$ when $w\to \infty$ $($see {\rm \cite[Section III.3]{bipmarxiv})} we have that $f\in \mathcal{V}_{\mathcal{S}}$, that is there exists $F\in \mathbb{V}_{P_0}$ such that $f=F\circ k_0$.
\end{Theorem}

\subsection{Separation of points}

The guiding principle of Dedekind--Weber theory is to reconstruct algebraically the Riemann surface from its function field, that in the case of a compact Riemann surface is the
field of meromorphic functions. A first fact to check is that we can separate points with functions. In the case of a compact Riemann surface the space of holomorphic
functions is reduced to constants, and it is useless. In our situation we can separate points using holomorphic functions in our structural ring.

\begin{Theorem}The ring $\hat{\mathcal{A}}_{\mathcal{S}}$ separates the points of $\mathcal{S}^*$.
\end{Theorem}

\begin{proof}Let $w_1, w_2\in \mathcal{S}^*$ with $w_1\not=w_2$. If both points are
regular points (non-ramification points), $w_1,w_2 \in \mathcal{S}$, take $z_1, z_2\in \mathbb{C}$
such that $z_i=k_0(w_i)$. Then the function $f\in \hat{\mathcal{A}}_{\mathcal{S}}$,
$f=F\circ k_0$, with{\samepage
\[
F(z)=(z-z_1){\rm e}^{P_0(z)}
\]
vanishes at $w_1$ but not at $w_2$.}

When one of the points, say $w_1$, is a
ramification point, then we can take $f=F\circ k_0$ with
\[
F(z)={\rm e}^{P_0(z)}
\]
the function $f$ will vanish at $w_1$ but not
at $w_2$. The function corresponding to ${\rm e}^{P_0}$ separates infinite ramification points from regular points.

The remaining case is when both points are ramification points $w_1,w_2 \in \mathcal{S}^*-\mathcal{S}$. Then, using Theorem \ref{thm:separation_directions} we have that there is a function $f_k$ that does not vanish simultaneously at both points, hence it separates $w_1$ and $w_2$.
\end{proof}

\looseness=-1 Dedekind--Weber theory in the case of the complex plane is elementary. Recall that to each point on $z_0\in \mathbb{C}$ we can associate a~maximal ideal $\mathfrak{m}_{z_0}$ of $\mathbb{C}[z]$, namely the ideal of functions vanishing at $z_0$. Conversely, any maximal ideal $\mathfrak{m}$ of $\mathbb{C}[z]$ is of this form since the residual field is~$\mathbb{C}$
\[
\mathbb{C}[z]/\mathfrak{m} \approx \mathbb{C}
\]
and $z$ is mapped by this quotient into some $z_0\in \mathbb{C}$, thus $\mathfrak{m}= \mathfrak{m}_{z_0}$. In that way the points of the complex plane $\mathbb{C}$ can be reconstructed algebraically from the ring of polynomials $\mathbb{C}[z]$, each point corresponding to a maximal ideal. The ring is of dimension $1$ and any prime ideal is maximal. In the same way we can reconstruct the Riemann sphere identifying points with discrete valuation rings in the field of fractions $\mathbb{C} (z)$.

In our situation, to each point of $\mathcal{S}^*$, including the infinite ramification points, we can associate a maximal ideal of $\hat{\mathcal{A}}_{\mathcal{S}}$.

\begin{Corollary}There is an embedding $\mathcal{S}^* \hookrightarrow \operatorname{Max} \hat{\mathcal{A}}_{\mathcal{S}}$, the space of maximal
ideals of $\hat{\mathcal{A}}_{\mathcal{S}}$, by
$w_0 \mapsto \mathfrak{m}_{w_0}$ where $\mathfrak{m}_{w_0}=\big\{f\in \hat{\mathcal{A}}_{\mathcal{S}} ;\, f(w_0)=0\big\}$.
\end{Corollary}

\begin{proof}Observe that any ideal $\mathfrak{m}_{w_0}$ is maximal because it is the kernel of the ring morphism $\hat{\mathcal{A}}_{\mathcal{S}} \to \mathbb{C}$,
\[
f\mapsto f(w_0)
\]
and
\[
\hat{\mathcal{A}}_{\mathcal{S}} / {\mathfrak{m}}_{w_0} \approx \mathbb{C}
\]
is a field, so ${\mathfrak{m}}_{w_0}$ is maximal.
\end{proof}

\begin{Proposition}The maximal ideal $\mathfrak{m}_{w^*}$ associated to an infinite ramification point is not principal.
\end{Proposition}

\begin{proof}Observe that ${\rm e}^{P_0}\circ k_0\in \mathfrak{m}_{w^*}$ and ${\rm e}^{P_0}\circ k_0$ has no non-trivial
divisors by Proposition \ref{prop:charac_exp}, hence $\mathfrak{m}_{w^*}$
is not principal.
\end{proof}

\subsection{Regular vs. infinite ramification points}

We define on $\hat{\mathbb{A}}_{P_0}$ the differential operator $D = \frac{{\rm d}}{{\rm d}z}$. The following lemma is clear.

\begin{Lemma}The ring $\hat{\mathbb{A}}_{P_0}$ endowed with $D$ is a differentiable ring. The ring of constants are the constant functions. Moreover,
the principal ideal generated by ${\rm e}^{P_0}$ is absorbent for the derivation:
\[
D(\hat{\mathbb{A}}_{P_0}) \subset \left ({\rm e}^{P_0}\right ).
\]
\end{Lemma}

The differential operator $D$ defines a derivation $\hat{\mathcal{D}}$ on the structural ring $\hat{\mathcal{A}}_{\mathcal{S}}$ which can be expressed on the variable $w=F_0(z)$ as
\[
\hat{\mathcal{D}} =\left ({\rm e}^{P_0}\circ k_0\right ) \frac{{\rm d}}{{\rm d}w}.
\]

\begin{Definition}The infinite ramification divisor is the principal ideal $\aleph_\infty$
generated by ${\rm e}^{P_0}\circ k_0$
\[
\aleph_\infty = \big({\rm e}^{P_0}\circ k_0\big).
\]
\end{Definition}

Next proposition is also clear.

\begin{Proposition}
We have that
\[
\hat{\mathcal{D}} (\hat{\mathcal{A}}_{\mathcal{S}}) \subset \aleph_\infty
\]
and $\aleph_\infty$ is the intersection of maximal ideals associated to infinite ramification points
\[
\aleph_\infty =\bigcap_{w^*} \mathfrak{m}_{w^*} .
\]
\end{Proposition}

Next theorem allows to distinguish regular and infinite ramification points from the position of their maximal ideal $\mathfrak{m}_{w_0}$ with respect to the ramification divisor $\aleph_\infty$.

\begin{Theorem}Let $\mathfrak{m}_{w_0}$ be the maximal ideal associated to a point $w_0\in \mathcal{S}^*$.
We have that $\mathfrak{m}_{w_0}\cap \hat{\mathcal{D}}^{-1}(\mathfrak{m}_{w_0})$ is a sub-ideal
of $\mathfrak{m}_{w_0}$, and
\begin{itemize}\itemsep=0pt
\item If $w_0\in \mathcal{S}$ is a regular point, we have that
$\mathfrak{m}_{w_0} \cap \aleph_\infty \not= \aleph_\infty$
and also,
$\mathfrak{m}_{w_0} \cap \aleph_\infty \not= \mathfrak{m}_{w_0}$
and $\mathfrak{m}_{w_0}\cap \hat{\mathcal{D}}^{-1}(\mathfrak{m}_{w_0})$ is a strict sub-ideal
of $\mathfrak{m}_{w_0}$,
$\mathfrak{m}_{w_0}\cap \hat{\mathcal{D}}^{-1}(\mathfrak{m}_{w_0}) \subsetneq \mathfrak{m}_{w_0}$.
\item If $w_0\in \mathcal{S}^*-\mathcal{S}$ is an infinite ramification point, then $\aleph_\infty \subset \mathfrak{m}_{w_0}$,
$\aleph_\infty \not= \mathfrak{m}_{w_0}$, and
$\mathfrak{m}_{w_0} \cap \aleph_\infty = \aleph_\infty$
also $\hat{\mathcal{D}}^{-1}(\mathfrak{m}_{w_0}) = \hat{\mathcal{A}}_{\mathcal{S}}$ so
$\mathfrak{m}_{w_0}\cap \hat{\mathcal{D}}^{-1}(\mathfrak{m}_{w_0}) =\mathfrak{m}_{w_0}$.
\end{itemize}
\end{Theorem}

\begin{proof}We prove that $\mathfrak{m}_{w_0}\cap \hat{\mathcal{D}}^{-1}(\mathfrak{m}_{w_0})$ is an ideal.
Let $f \in \mathfrak{m}_{w_0}\cap \hat{\mathcal{D}}^{-1}(\mathfrak{m}_{w_0})$. We check that
$f.h \in \mathfrak{m}_{w_0}\cap \hat{\mathcal{D}}^{-1}(\mathfrak{m}_{w_0})$ for any $h\in \hat{\mathcal{A}}_{\mathcal{S}}$.
We have
\begin{gather*}
f(w_0)=0, \qquad \hat{\mathcal{D}} (f)(w_0)=0,
\end{gather*}
so we get $(f.h)(w_0)=0$ and
\[
\hat{\mathcal{D}} (fh)(w_0) = \hat{\mathcal{D}} (f)(w_0).h(w_0)+f(w_0).\hat{\mathcal{D}} (h)(w_0)=0.
\]

When $w_0$ is an infinite ramification point, it is clear that $\aleph_\infty \subset \mathfrak{m}_{w_0}$ and
$\aleph_\infty \not= \mathfrak{m}_{w_0}$ because $d\geq 2$. Taking preimages in $\aleph_\infty \subset \mathfrak{m}_{w_0}$
we get
\[
\hat{\mathcal{A}}_{\mathcal{S}} \subset \hat{\mathcal{D}}^{-1}(\mathfrak{m}_{w_0}),
\]
thus $\hat{\mathcal{D}}^{-1}(\mathfrak{m}_{w_0}) = \hat{\mathcal{A}}_{\mathcal{S}}$.

When $w_0\in \mathcal{S}$ is a regular point, we have ${\rm e}^{P_0}\circ k_0\notin \mathfrak{m}_{w_0}$, so ${\rm e}^{P_0}\circ k_0\in \aleph_\infty-\mathfrak{m}_{w_0}$. Also, there are functions $f \in \mathfrak{m}_{w_0}-\aleph_\infty$. For example, one can choose $f=F\circ k_0$ where $F$ is a~linear combination of $1, F_0,\dots, F_{d-1}$ vanishing at $z_0=k_0(w_0)$ (codimension $1$ condition) and not a multiple of ${\rm e}^{P_0}$ (another codimension $1$ condition by the non-vanishing of the Ramificant determinant),
then not all asymptotic values of~$F$ can be~$0$ because otherwise~$F$ would be a~multiple of ${\rm e}^{P_0}$ by Theorem~\ref{thm:vanishing_periods}.
\end{proof}

\subsection*{Acknowledgements}
We are grateful to Y.~Levagnini and the referees for their careful reading and corrections that improved the article.

\pdfbookmark[1]{References}{ref}
\LastPageEnding

\end{document}